\documentclass[12pt,english]{article}
\usepackage{babel}
\usepackage[cp1250]{inputenc}
\usepackage[T1]{fontenc}
\usepackage{amsfonts}
\usepackage{amsmath}
\usepackage{amsthm}
\usepackage{hyperref}
\usepackage{enumerate}
\usepackage{graphicx}
\usepackage{pictex}
\usepackage{color}

\newcommand{\eqq}[2]{\begin{equation}  #1  \label{#2} \end{equation}    }

\newcommand{\hd}{\hspace{0.2cm}}

\newcommand{\no}{\noindent}

\newcommand{\N}{\mathbb{N}}

\newcommand{\en}{\mathbb{N}}

\newcommand{\rr}{\mathbb{R}}

\newcommand{\pt}{\frac{d}{dt}}

\newcommand{\da}{D^{\alpha}}

\newcommand{\das}{D^{\alpha}_{s}}

\newcommand{\ia}{I^{\alpha}}

\newcommand{\ias}{I^{\alpha}_{s}}

\newcommand{\ijas}{I^{1-\alpha}_{s}}

\newcommand{\qs}{Q_{s,T}}

\newcommand{\qst}{Q_{s,t}}

\newcommand{\ga}{\Gamma(\alpha)}

\newcommand{\gja}{\Gamma(1-\alpha)}

\newcommand{\jga}{\frac{1}{\ga}}

\newcommand{\jgja}{\frac{1}{\gja}}

\newcommand{\gaa}{\Gamma(\alpha)\Gamma(1-\alpha)}

\newcommand{\jgaa}{\frac{1}{\gaa}}

\newcommand{\ca}{c_{\alpha}}

\newcommand{\ja}{}

\newcommand{\inb}{\int_{0}^{b}}

\newcommand{\izt}{\int_{0}^{t}}

\newcommand{\izT}{\int_{0}^{T}}

\newcommand{\izj}{\int_{0}^{1}}

\newcommand{\izta}{\int_{0}^{\tau}}

\newcommand{\itt}{\int_{t_*}^{t}}

\newcommand{\itta}{\int_{t_*}^{\tau}}

\newcommand{\ta}{(t-\tau)^{-\alpha}}

\newcommand{\tamj}{(t-\tau)^{\alpha-1}}

\newcommand{\tap}{(\tau-p)^{-\alpha}}

\newcommand{\tapj}{(\tau-p)^{-\alpha-1}}

\newcommand{\taj}{(t-\tau)^{-\alpha-1}}

\newcommand{\uxt}{u(x,\tau)}

\newcommand{\utt}{u_{\tau}(x,\tau)}

\newcommand{\dt}{d\tau}

\newcommand{\ith}{\int_{0}^{t-h}}

\newcommand{\itht}{\int_{t-h}^{t}}

\newcommand{\ibh}{\inb \ith}

\newcommand{\ibth}{\int_{b}^{s(t-h)}}

\newcommand{\isth}{\int_{s^{-1}(x)}^{t-h}}

\newcommand{\ibsh}{\ibth \isth}

\newcommand{\sj}{s^{-1}(x)}

\newcommand{\izs}{\int_{0}^{s(t)}}

\newcommand{\izsT}{\int_{0}^{s(T)}}

\newcommand{\izst}{\int_{0}^{s(\tau)}}

\newcommand{\snm}{\sum_{n=0}^{m}}

\newcommand{\cnmt}{c_{n,m}(t)}

\newcommand{\vn}{\varphi_{n}}

\newcommand{\vk}{\varphi_{k}}

\newcommand{\vkx}{\varphi_{k,x}}

\newcommand{\lan}{\lambda_{n}}

\newcommand{\lk}{\lambda_{k}}

\newcommand{\cln}{\cos{\lan(t)x}}

\newcommand{\qds}{\sqrt{\frac{2}{s(t)}}}

\newcommand{\nbc}{\| \bc \|_{X(t)}}

\newcommand{\nbcj}{\| \bc \|_{X(T_{1})}}

\newcommand{\ccT}{\| c \|_{X(T)}}

\newcommand{\ep}{\varepsilon}

\newcommand{\gem}{g^{\ep}_{m}}

\newcommand{\get}{G^{\ep}(t)}

\newcommand{\df}{\widetilde{D}}

\newcommand{\dos}{\dot{s}}

\newcommand{\iztj}{\int_{0}^{t_{1}}}

\newcommand{\iztd}{\int_{0}^{t_{2}}}

\newcommand{\iztjd}{\int_{t_{1}}^{t_{2}}}

\newcommand{\izjjd}{\int_{\frac{t_{1}}{t_{2}}}^{1}}

\newcommand{\tamjj}{(t_{1}-\tau)^{\alpha-1}}

\newcommand{\tamjd}{(t_{2}-\tau)^{\alpha-1}}

\newcommand{\nf}{\| f \|}

\newcommand{\tadj}{t_{2}^{\alpha}-t_{1}^{\alpha}}

\newcommand{\tdj}{t_{2}-t_{1}}

\newcommand{\tdja}{(\tdj)^{\alpha}}

\newcommand{\tdjaj}{(\tdj)^{\alpha-1}}

\newcommand{\xtj}{X(T_{1})}

\newcommand{\bc}{\bar{c}}

\newcommand{\pf}{\tilde{P}}

\newcommand{\vtm}{\widetilde{v}_{m}  }

\theoremstyle{plain}
\newtheorem{remark}{Remark}[section] 	

\newtheorem{corollary}{Corollary}[section]	

\newtheorem{theorem}{Theorem}[section] 	

\newtheorem{lemma}{Lemma}[section]	
\theoremstyle{definition}
\newtheorem{definition}{Definition}[section]

\numberwithin{equation}{section}

\newcommand{\kom}{}  

\def\cP{\mathcal P}

\begin{document}
\title{Weak solutions of fractional differential equations in non cylindrical domains}

\author{A. Kubica${}^1$, P. Rybka${}^2$, K. Ryszewska${}^1$\\
\medskip\\
${}^1$
Department of Mathematics and Information Sciences\\
Warsaw University of Technology\\
ul. Koszykowa 75, 00-662 Warsaw, Poland\\e-mail:
{\tt A.Kubica@mini.pw.edu.pl}
\medskip\\
${}^2$ Institute of Applied Mathematics and Mechanics,
Warsaw University\\ ul. Banacha 2, 02-097 Warsaw, Poland\\
e-mail: {\tt rybka@mimuw.edu.pl}}

\maketitle
\begin{quote}
 \footnotesize
{\bf Abstract.} We study a time fractional heat equation in a non cylindrical domain. The problem is one-dimensional. We prove existence of properly defined weak solutions by means of the Galerkin approximation.
\end{quote}



\section{Introduction}
\subsection{Motivation}
In this paper we study the heat equation with the Caputo time derivative in a non cylindrical domain. It is intended to be the first in a series devoted the Stefan problem with fractional derivatives. We address here a simple, yet non-trivial question of the existence of solutions to the problem where the interfacial curve is given to us.

The interest in fractional PDE's stems from many sources. One of them is the theory of stochastic processes admitting jumps and continuous paths, see \cite{poincare}, \cite{bkm}, \cite{klafter}. Our motivation comes from phenomenological models of sediment transport, see \cite{Foufoula}, \cite{Ganti}, \cite{Paola}, \cite{Schumer}. Apparently, this problem awaits a systematic treatment.

In these problems, the position of the advancing front $s$ is not known, i.e., it is a part of the problem. Here, we consider a simplified situation,
\eqq{ \left\{
\begin{array}{rclcc}
\das u (x,t) &=& u_{xx}(x,t)+f(x,t) & \hd \mbox{ for } \hd &  0<x<s(t), \hd 0<t<T , \\
- u_{x}(0,t)&=&h(t)  & \hd \mbox{ for } \hd &  0<t<T,\\
 u(s(t),t)&=&0 & \hd \mbox{ for } \hd &  0<t<T,\\
u(x,0)&=&u_{0}(x) & \hd \mbox{ for } \hd & 0<x<b=s(0), \\
\end{array}
\right.
}{a}
where $h(t)$, $f(x,t)$, $u_{0}(x)$, $s$ are given functions and $s$ is nondecreasing.
By $\das$ we understand the Caputo fractional derivative, defined as follows
\eqq{\das w (x,t)= \left\{  \begin{array}{lll}
\jgja  \int_{0}^{t} (t-\tau)^{-\alpha}w_{t}(x,\tau) d \tau &  \mbox{ for } & x\leq s(0), \\ \\
\jgja  \int_{s^{-1}(x)}^{t} (t-\tau)^{-\alpha}w_{t}(x,\tau) d \tau &  \mbox{ for } & x>s(0).\\
 \end{array}
 \right. }{aa}

In the original free boundary problem the above system is augmented by an equation governing the evolution of $s$,
here $s$ is given to us. We assume that $s$ is increasing.
We will also use the following convention: if $x>b$, then $\sj = \max\{t: \hd s(t)=x \}$ and if $x\in [0,b] $, then $\sj =0$.

Experience with fractional derivatives tells us that tools used for PDE's require substantial modification before they can be applied to equations with fractional derivatives. In case of time fractional derivatives this means that we have to take into account the whole history of the process. This is the main difficulty of the analysis. This is why we set limited goals in this paper.

Our main result, expressed in Theorem \ref{main} below, is the existence of properly defined weak solutions to   (\ref{a}). We use for this purpose the Galerkin method. This is a  straightforward approach in case of PDE's, however, here it gets complicated when applied to (\ref{a}). Now, we briefly describe the content of our paper. In Section 2 we construct approximate solutions by means of the Galerkin method. In Section 3 we derive the necessary estimate and show that we can pass to the limit. This process yields a existence of a weak solution. Proofs of technical results are presented in the Appendix.

\subsection{Preliminaries and a definition of a weak solution}

Here, we will make our assumptions,  upheld throughout this paper. We will denote by $\alpha \in (0,1)$ the order of the Caputo derivative. By $s$ we denote the position of the interface. Its initial position $b=s(0)$ is positive. Function $s$ is not only increasing and continuous, but also
\eqq{ t \mapsto t^{1-\alpha}\dot{s}(t)\in C([0,T])\qquad \mbox{ and }\qquad \dot s \ge 0
.}{zalA}
With the help of $s$, we define a non cylindrical domain $\qst$ by the following formula,
\[
\qst=\{(x,\tau): \hd 0<x<s(\tau), \hd 0<\tau<t \}.
\]

The fractional derivative $D_s^\alpha$ is a complicated operator. In order to simplify the analysis, we introduce an auxiliary integral operator for functions
defined on the domain $\qst$ by the following formula,
\eqq{\ias w (x,t)= \left\{  \begin{array}{lll}
\jga  \int_{0}^{t} (t-\tau)^{\alpha-1}w(x,\tau) d \tau &  \mbox{ for } & x\leq s(0), \\ \\
\jga  \int_{s^{-1}(x)}^{t} (t-\tau)^{\alpha-1}w(x,\tau) d \tau &  \mbox{ for } & x>s(0).\\
 \end{array}
 \right. }{ab}
It is easy to verify that if $u(x,t) $ is an absolutely continuous function with respect to $t$ variable and it satisfies the boundary and initial conditions of (\ref{a}), then the following identity holds
\eqq{\das u(x,t)=\pt \ijas [u(x,t)-\tilde{u}_{0}(x)],}{ha}
where ${\tilde{u} }_{0}(x)$ will always denote the extension of $u_{0}(x)$ by zero. 
This equality suggests a weak form of the system (\ref{a}).

\begin{definition}\rm  Let us
assume that $f\in L^{2}(\qs)$, $h\in C^{1}([0,T])$ and $u_{0}\in L^{2}(0,b)$. A function  $u$ is \textit{a weak solution} of (\ref{a}) if $u,u_{x}\in L^{2}(\qs)$, $I^{1-\alpha}  \widetilde{u} (\cdot , t) \in L^{\infty}(0,T;L^{2}(0,s(T)))$
and  $u$ fulfills the identity
\[
-\int_{\qs}\ijas [u(x,t)-\tilde{u}_{0}(x)] \varphi_{t}(x,t) dxdt+\int_{\qs} u_{x}(x,t)\varphi_{x}(x,t)dxdt
\]
\eqq{=\int_{0}^{T} h(t)\varphi(0,t)dt+ \int_{\qs}f(x,t) \varphi(x,t)dxdt }{hb}
for all  $\varphi \in C^{1}(\overline{\qs})$ such that $\varphi(x,T)=0$ for  $x\in [0,s(T)]$ and  $\varphi(s(t),t)=0$ for $t\in [0,T]$.
\label{weakdef}
\end{definition}
\no Now we are ready to formulate our main result.
\begin{theorem} Let us
assume that $\alpha \in (0,1)$, $b>0$ and $s$ satisfies (\ref{zalA}). If $f\in L^{2}(\qs)$, $h\in C^{1}([0,T])$ and $u_{0}\in L^{2}(0,b)$ then, there exists  a weak solution of (\ref{a}).
\label{main}
\end{theorem}

The proof will be provided in Section \ref{s:w}. In its
first step we reduce the problem to the case of zero boundary condition of (\ref{a}). For this purpose we fix a smooth function $\eta=\eta(x) $ such that $\eta'(0)=1$ and  $\eta(x)=0$ for $x\geq b$. Then   $v(x,t):=u(x,t)-h(t)\eta(x)$ satisfies
\eqq{ \left\{
\begin{array}{rclcc}
\das v (x,t) &=& v_{xx}(x,t)+g(x,t) & \hd \mbox{ in } \hd &  \qs, \\
 v_{x}(0,t)&=&0  & \hd \mbox{ for } \hd &  0<t<T,\\
 v(s(t),t)&=&0 & \hd \mbox{ for } \hd &  0<t<T,\\
v(x,0)&=&v_{0}(x) & \hd \mbox{ for } \hd & 0<x<b, \\
\end{array}
\right.
}{hc}
where
\eqq{g(x,t)=f(x,t ) +h(t)\eta_{xx}(x)-\eta(x)\da h(t), \hd v_{0}=u_{0}-h(0)\eta(x).}{wd}
\no Therefore, we shall show that there exists $v(x,t)$ such that the identity
\eqq{-\int_{\qs}\ijas [v(x,t)-\tilde{v}_{0}(x)] \varphi_{t}(x,t) dxdt+\int_{\qs} v_{x}(x,t)\varphi_{x}(x,t)dxdt=\int_{\qs}g(x,t) \varphi(x,t)dxdt, }{hbb}
holds for each $\varphi$ as in the definition~\ref{weakdef}, where ${\tilde{v} }_{0}(x)$ is a continuation by zero of $v_{0}(x)$.\\

The proof of this Theorem is based on the Galerkin method. In Section \ref{s:a} we construct an approximate  solution. In Section \ref{s:w} we pass to the limit, after having derived necessary a priori estimates in Lemmas \ref{stwjeden} and \ref{oszac}. Many auxiliary facts used in our analysis are collected in the Appendix.

\section{Approximate solutions}\label{s:a}
We are looking for  approximate solutions of (\ref{hbb}) in the following form
\eqq{v_{m}(x,t) = \snm \cnmt \vn(x,t),   }{ec}
where for each $t\in [0,T]$ functions $\{ \vn (\cdot, t)\}_{n\in \en}$ constitute an orthonormal basis in $L^{2}(0,s(t))$. More precisely, we  define functions $\vn$ as follows:
\eqq{
- \varphi_{n,xx}(x,t)=\lambda^{2}_{n}(t)\vn(x,t), \hd  \varphi_{n,x}(0,t)=0, \hd \vn(s(t),t)=0.
}{defvn}
These conditions imply that $\vn(x,t)= \qds \cln$, where
\begin{equation}\label{def_lambdy}
 \lan(t)= \frac{\pi}{s(t)}[\frac{1}{2}+n].
\end{equation}
We also have
\eqq{\int_{0}^{s(t)} \vn(x,t)\cdot \vk(x,t) dx = \delta_{k,n}, }{ed}
\eqq{\int_{0}^{s(t)} \varphi_{n,x}(x,t) \cdot \varphi_{k,x}(x,t) dx = \lan^{2}(t)\delta_{k,n}. }{ee}
Subsequently, we define vector functions  $c_{m}(t)= (c_{0,m}(t), \dots , c_{m,m}(t))$ as solutions to the following system
\eqq{\izs \das v_{m}(x,t) \cdot \vk(x,t) dx + \izs v_{m,x}(x,t) \cdot \vkx(x,t) dx=\izs \gem(x,t) \vk(x,t) dx,  \hd t\in (0,T]  }{ef}
\eqq{
c_{k,m}(0)= \inb v_{0}(x)\vk(x,0)dx, \hd  \hd \hd k=0,\dots, m,
}{eff}
where
\[
\gem(x,t)= \snm \izs g^{\ep}(y,t)\vn(y,t)dy\vn(x,t).
\]
We assume that $g^{\ep}(x,t)$ is a smooth (with respect to time variable) approximation of  $g(x,t)$ \kom obtained by convolution of  mollifier and $\widetilde{g} $.
In this section we shall prove the following result.

\begin{theorem}\label{przyb}
Let us
assume that $s$ satisfies (\ref{zalA}) and $\alpha \in (0,1)$, $b>0$, $T>0$, in addition $f\in L^{2}(\qs)$, $h\in C^{1}([0,T])$ and $u_{0}\in L^{2}(0,b)$. Then, for each  $m\in \N $ there exists  $c_{m}(\dot)\in AC([0,T];\rr^{m}) $, satisfying (\ref{eff}), such that  $v_{m}$ given by  (\ref{ec}) satisfies  (\ref{ef}). Furthermore, $t^{1-\alpha} c_{m}'\in C([0,T];\rr^{m})$ and $t^{1-\alpha}v_{m,t} $ is continuous on $\overline{\qs} $.
\end{theorem}

\begin{remark}\rm
System (\ref{ef}) can be written in a form suitable for $c_{m}(t)$, which are only continuous. Then, the proof of the existence of continuous solutions $c_{m}(t)$ is much simpler than presented below because it is sufficient to assume absolute continuity of $s$. However, in this case,  we are not able to deduce absolute continuity of solutions $c_{m}(t)$. As a result, we are not able to obtain a priori estimates for an approximate solution (see Lemma~\ref{stwjeden}). In order to obtain them, we were forced to assume that $t^{1-\alpha}\dot{s}(t) $ is continuous and we are looking for $c_{m}(t)$ in a special function space.
\end{remark}

The proof of Theorem~\ref{przyb} will be divided into several steps, each is contained in a separate subsection. At the first stage, see \S \ref{ss:1}, we set up the auxiliary integral equation for the appropriate solutions. Next, in \S 2.2 we find a suitable function space. In \S 2.3 we show existence of local in time solution to the system introduced in \S \ref{ss:1}. Finally, in \S 2.4 we show its global solvability.

\paragraph{Notation}
We
denote by $c_{0}$  a generic constant depending on $b$, $m$, $\alpha$, $T$, $\max\{ |s(t)|:\ t\in [0,T] \}$ and  $\max\{  t^{1-\alpha} \dot{s} (t): t\in [0,T] \} $.

\subsection{Reformulation of system (\ref{ef})}\label{ss:1}
In order to streamline the argument and to make it transparent, we moved a number of estimates to the Appendix. We also recall there the definition of the Euler Beta function, (\ref{dodb}), and its relation $\Gamma$ function, (\ref{dodc})

We note that once we assume (\ref{zalA}), then for each $x$ functions $\vn(x,\cdot)$ are absolutely continuous.
Under the assumption that $c_{m}$ are absolutely continuous, the system (\ref{ef}) with $v_{m}$ given by (\ref{ec}) can be equivalently written as follows,
\[
\snm \izs \das \left[\cnmt \vn(x,t) \right]\cdot \vk(x,t) dx + \lk^{2}(t) c_{k,m}(t)=\izs g^{\ep}(x,t)\vk(x,t)dx,
\]
where we used (\ref{ed}) and (\ref{ee}). Thus,
\[
\snm \izt \izst \ta c_{n,m}'(\tau) \cdot \vn(x,\tau) \cdot \vk (x,t) dx \dt
\]
\[
+\snm \izt \izst \ta c_{n,m}(\tau) \cdot \left( \vn(x,\tau) \right)_{\tau}  \cdot \vk(x,t) dx \dt
\]
\[
+\gja \lk^{2}(t) c_{k,m}(t)=\gja\izs g^{\ep}(x,t)\vk(x,t)dx.
\]
Hence, making use of (\ref{ed}), we obtain
\[
 \izt  \ta c_{k,m}'(\tau) \dt
+\snm \izt  \ta c_{n,m}'(\tau) \izst \vn(x,\tau)\left[ \vk(x,t) -\vk(x,\tau) \right]   dx \dt
\]
\[
+\snm \izt  \ta c_{n,m}(\tau) \izst \left( \vn(x,\tau) \right)_{\tau}  \cdot \vk(x,t) dx \dt
\]
\[
+\gja \lk^{2}(t) c_{k,m}(t)=\gja\izs g^{\ep}(x,t)\vk(x,t)dx.
\]
In order to make further computation more transparent, we introduce the following notation,
\[
B(\tau, t)_{n,k}= \izst \vn(x,\tau)\left[ \vk(x,t) -\vk(x,\tau) \right] dx \hd \hd \mbox{for } \hd 0\leq \tau \leq t \leq T,
\]
\[
D(\tau, t)_{n,k}= \izst \left( \vn(x,\tau)\right)_{\tau}  \cdot \vk(x,t) dx \hd \hd \mbox{for } \hd 0\leq \tau \leq t \leq T,
\]
\[
E(t)_{n,k}=\gja \lk^{2}(t) \delta_{n,k} \hd \hd \mbox{for } \hd 0\leq  t \leq T,
\]
\[
G^{\ep}(t)_{k}=\gja\izs g^{\ep}(x,t)\vk(x,t)dx \hd \hd \mbox{for } \hd 0 \leq t \leq T.
\]
With this notation
system (\ref{ef}) may  be written in a more compact way as follows,
\eqq{
\izt \ta c_{m}'(\tau) \dt+ \izt \ta B(\tau, t )c_{m}'(\tau) \dt +
\izt \ta D(\tau, t) c_{m}(\tau ) \dt + E(t) c_{m}(t)=G^{\ep}(t).
}{defdo}
After having applied  $\ga I^{\alpha} $ to both sides of (\ref{defdo}) and using (\ref{doda}), we arrive at
\[
c_{m}(t)=c_{m}(0)-\ca \izt \tamj \izta\tap B(p, \tau )c_{m}'(p) dp \dt
\]
\eqq{ -\ca\izt \tamj \izta \tap D(p,\tau) c_{m}(p ) dp \dt -  \ca\izt \tamj   E(\tau) c_{m}(\tau)\dt+\jgja \ia \get,}{ehhh}
 where
 \begin{equation}\label{calfa}
 \ca=\jgaa .
 \end{equation}
We 
decompose function $D$, which will allow us to obtain useful estimates.
\[
D(\tau, t)_{n,k} 
=\izst \left( \vn(x,\tau) \right)_{\tau}  \cdot \vk(x,\tau) dx
+\izst \left( \vn(x,\tau) \right)_{\tau}  \cdot \left[\vk(x,t)- \vk(x,\tau) \right] dx.
\]
Using (\ref{defvn}) and (\ref{ed}), we have
\[
\izst \left( \vn(x,\tau) \right)_{\tau}  \cdot \vn(x,\tau) dx
=\frac{1}{2} \left[ \underbrace{\izst  \vn^{2}(x,\tau)  dx}_{=1}
\right]_{\tau}
-\frac{\dot{s}(\tau)}{2}   \underbrace{\vn^{2}(s(\tau),\tau)}_{=0} =0.
\]
For  $k \not = n$, we get
\[
\izst \left( \vn( x,\tau) \right)_{\tau}  \cdot \vk(x,\tau) dx=  \frac{\dot{s}(\tau)}{s(\tau)}\hat{D}_{n,k}  ,
\]
for a constant matrix $\hat{D}$. Therefore, we have
\eqq{
D(\tau, t)_{n,k}= \frac{\dot{s}(\tau)}{s(\tau)}\hat{D}_{n,k} +\widetilde{D}(\tau, t)_{n,k},
}{rozkD}
where
\eqq{
\widetilde{D}(\tau, t)_{n,k}=\izst \left( \vn(x,\tau) \right)_{\tau}  \cdot \left[\vk(x,t)- \vk(x,\tau) \right] dx.
}{Dtil}
Taking into account (\ref{rozkD}) and the properties of the Beta function,
we can write system (\ref{ehhh}) in this way
\[
c_{m}(t)=c_{m}(0)-\ca \izt \tamj \izta\tap B(p, \tau )c_{m}'(p) dp \dt  -\izt  \frac{\dot{s}(p)}{s(p)} \hat{D} c_{m}(p)dp
\]
\eqq{ -\ca\izt \tamj \izta \tap \widetilde{D}(p,\tau) c_{m}(p ) dp \dt -  \ca\izt \tamj   E(\tau) c_{m}(\tau)\dt+\jgja \ia \get.}{eh}
Our goal is to show that integral equation (\ref{eh}) has an absolutely continuous solution $c_m$, such that $t^{1-\alpha}c_{m}'(t)$ is continuous.
\begin{remark}\rm
If $s$ is constant, i.e. $s(t)\equiv b$ (in this case $\qs$ is a cylindrical domain), then equation (\ref{eh}) significantly simplifies,  because in this situation $B\equiv 0$, $\widetilde{D}\equiv 0$, $\dot{s}=0$ and $E$ is a constant matrix.
\end{remark}

We shall prove the existence of a solution of (\ref{eh}) by applying the Banach fixed point theorem on an appropriate space. For this purpose we collect here  estimates of $B$,  $\widetilde{D}$ and $E$, which we will use later on.

\begin{corollary}\label{c:d1}
Let us assume that $s\in AC([0,T])$, $\dot{s}\geq 0$ and $s(0)=b$, then
\eqq{|B(\tau, t)| \leq c_{0}[s(t)-s(\tau)],}{eg}
\eqq{|B_{\tau}(\tau, t)| \leq c_{0}\dot{s}(\tau),}{egp}
\eqq{|B_{t}(\tau, t)| \leq c_{0}\dot{s}(t),}{egt}
\eqq{|B(\tau,t_{2})-B(\tau, t_{1})| \leq  c_{0} |s(t_{2} ) - s(t_{1})|,}{kc}
\eqq{|\widetilde{D}(\tau , t )| \leq c_{0} \dot{s} (\tau) [s(t)- s(\tau)],}{ja}
\eqq{|\widetilde{D}_{t}(\tau , t )| \leq c_{0} \dot{s} (\tau) \dot{s}(t),}{jb}
\eqq{|E'(t)| \leq c_{0} \dos(t),}{oi}
hold a.e., where $c_0$ depends only on $m,b,\alpha, \| t^{1-\alpha}\dot{s} \|_{C[0,T]} $  and $T$.
\end{corollary}
\begin{proof}
 We reach these conclusions by  simple calculations.
\end{proof}

\begin{corollary}\label{c:d2}
 We assume that $s\in AC([0,T])$, $s(0)=b$ and condition (\ref{zalA}) holds, then
 \eqq{|s(\tau)-s(p) | \leq c_{0}|\tau^{\alpha}-p^{\alpha}|\leq c_{0}|\tau-p|^{\alpha}, }{ib}
\eqq{|s(\tau)-s(p)| \leq c_{0}p^{\alpha-1}|\tau-p|, \hd \mbox{ for } p<\tau. }{od}
\end{corollary}
\begin{proof}
 Also in this case  by  inspection we establish validity of these statements.
\end{proof}

We draw further conclusions from Corollaries \ref{c:d1} and \ref{c:d2}. Namely,
(\ref{eg}), (\ref{ja}) and (\ref{od}) imply that
\eqq{\lim_{t \rightarrow \tau^{+}} (t-\tau)^{-\alpha} B (\tau,t) = 0 \hd \mbox{ for } \tau>0,   }{of}
\eqq{\lim_{t \rightarrow \tau^{+}} (t-\tau)^{-\alpha} \df (\tau,t) = 0 \hd \mbox{ for } \tau>0.   }{og}

\subsection{Our choice of the function space}
After these preparations we define a function space, where we shall look for a solution to (\ref{eh}). For $T>0$ and $c_m(0)$ as in (\ref{eff}) we set,
\eqq{
X(T)= \{c\in C^{1}((0,T];\mathbb{R}^{m}): \hd c(0)=c_{m}(0), \hd t^{1-\alpha}c'(t)\in C([0,T];\mathbb{R}^{m}) \}.
}{defXT}
It is easy to check that this space endowed with the natural  norm
\[
\| c \|_{X(T)} = \| c \|_{C([0,T])}+ \| t^{1-\alpha} c'\|_{C([0,T])}
\]
is a Banach space.

\begin{remark}\rm Having introduced the norm of $X(T)$, we may write that by $c_0$ we denote a generic constant
$c_0 = c_0(b,m, \alpha, T, \|s\|_{X(T)})$. 
\end{remark}

For each
$c\in X(T)$, we define a function
\[
(Pc)(t)=c_m(0)-(P_{1}c)(t)-(P_{2}c)(t)-(P_{3}c)(t)-(P_{4}c)(t)+\jgja \ia \get,
\]
where
\[
(P_{1}c)(t)=\ca \izt \tamj \izta\tap B(p, \tau )c'(p) dp \dt,
\]
\[
(P_{2}c)(t)= \izt  \frac{\dot{s}(p)}{s(p)} \hat{D} c(p)dp,
\]
\[
(P_{3}c)(t)=\ca\izt \tamj \izta \tap \widetilde{D}(p,\tau) c(p ) dp \dt,
\]
\[
(P_{4}c)(t)=  \ca\izt \tamj   E(\tau) c(\tau)\dt
\]
and $c_\alpha$ is given by (\ref{calfa}).
With this definition of $P$, we can concisely write (\ref{eh}) as
\begin{equation}\label{r:ca}
 c_m(t) = c_m(0) + (P c_m)(t), \qquad t\in [0,T].
\end{equation}

We shall see that $P$ maps $X(T)$ into itself. This is done in Lemma  \ref{defX} below. Then, in next subsection,
we shall show that for a sufficiently small $T$ operator  $P$ will be a contraction. Thus, by Banach Theorem, operator  $P$ 
has a unique fixed point.

In order to proceed, we state estimates frequently used to establish Lemma  \ref{defX}. They are
shown in the Appendix.

\begin{lemma}
If $f\in AC[0,T]$, then $(\ia f)(t)\in AC[0,T]$ and $(\ia f)'(t)=\ia f'(t) + t^{\alpha-1}\frac{f(0)}{\ga}$.
\label{lemone}
\end{lemma}

\begin{lemma}
Let us assume that $p\mapsto p^{1-\alpha}w(p)\in L^{\infty}(0,T)$ and let
$$
g_{2}(\tau) =\izta \tap B (p,\tau) w(p) dp .
$$
Then, there exists an absolutely continuous representative of integrable function $g_2$. That is, there exists  $\tilde{g}_{2} \in AC[0,T]$ satisfying $\tilde{g}_{2}(0)=0$ and such that
$g_{2}=\tilde{g}_{2}$ a.e. on $[0,T]$.
\label{lemtwod}
\end{lemma}

Now, we are ready for the main result of this subsection.
\begin{lemma}
Let us take any $T>0$ and $c\in X(T)$, then  $Pc\in X(T)$.
\label{defX}
\end{lemma}
\begin{proof}
If
$c\in X(T)$, then we set
$$
\cP_i(t) = t^{1-\alpha} \left(P_{i}c  \right)'(t),\qquad i=1,\ldots,4.
$$
At first, we will show that
 \eqq{
 \cP_1, \cP_2, \cP_3, \cP_4
\in C([0,T];\rr^{m}). }{qj}
We consider each  $\cP_i$, $i=1,\ldots,4$ separately. However, we first record  for all elements $c\in X(T)$ the following bound
\eqq{|c'(p)|\leq \ccT p^{\alpha-1}.  }{pa}
\paragraph{Case of $\cP_{1}$.}
Our goal is to prove that
\eqq{
\cP_{1} \in C([0,T];\rr^{m}). }{qh}
 If we keep in mind Lemma~\ref{lemone} and  Lemma~\ref{lemtwod}, then it is enough to show  that
\eqq{t\mapsto \cP_{1,1}(t):=t^{1-\alpha} \izt \tamj \left[ \izta \tap B (p,\tau)  c'(p)dp \right]_{\tau} \dt \in C([0,T];\rr^{m}).}{poh}
In order to prove (\ref{poh}), we 
consider the difference, where  $t_{1}<t_{2}$,
\begin{eqnarray*}
 &&t^{1-\alpha}_{2} \iztd \tamjd \left[ \izta \tap B (p,\tau) c'(p)dp \right]_{\tau} \dt \\
 &&\quad -t^{1-\alpha}_{1} \iztj \tamjj \left[ \izta \tap B (p,\tau) c'(p)dp \right]_{\tau} \dt\\
 &=&t^{1-\alpha}_{2} \iztjd \tamjd \left[ \izta \tap B (p,\tau) c'(p) dp \right]_{\tau} \dt\\
 &&+t^{1-\alpha}_{2} \iztj [\tamjd -\tamjj] \left[ \izta \tap B (p,\tau) c'(p) dp \right]_{\tau} \dt \\
 &&\quad+(t^{1-\alpha}_{2} - t^{1-\alpha}_{1})\iztj \tamjj \left[ \izta \tap B (p,\tau) c'(p)  dp \right]_{\tau} \dt \\
 &\equiv& W_{1}(t_{1},t_{2})+W_{2}(t_{1},t_{2})+W_{3}(t_{1},t_{2}).
\end{eqnarray*}
By definition,  in the case of $t_{1}=0$ we have $W_{2}(0,t_{2})=0$ and $W_{3}(0,t_{2})=0$,  hence continuity of $\cP_{1,1}$ at $t=0$ will be shown if
$ \lim_{t_{2}\rightarrow 0 }W_{1}(0,t_{2})=0$.
For this purpose, we estimate $W_{1}$. With the help of (\ref{of}) we perform the differentiation with respect to $\tau$ to get,
\begin{eqnarray*}
&W_{1}(t_{1},t_{2})&=
t^{1-\alpha}_{2} \iztjd \tamjd \left[ \izta  \tap B (p,\tau) c'(p) dp \right]_{\tau} \dt\\
&
&=-\alpha t^{1-\alpha}_{2} \iztjd \tamjd  \izta \tapj B (p,\tau) c'(p)dp \dt
\\
&
&\quad+t^{1-\alpha}_{2} \iztjd \tamjd  \izta \tap B_{\tau} (p,\tau) c'(p) dp \dt \\
&&\equiv - \alpha W_{1,1}(t_{1},t_{2})+ W_{1,2}(t_{1},t_{2}).
\end{eqnarray*}
After having introduced the following functions,
\eqq{Q_{1,1}(t_{1},t_{2})=t^{1-\alpha}_{2} \iztjd \tamjd  \izta (\tau - p )^{-\alpha-1} [s(\tau)- s(p)] p^{\alpha-1}dp \dt,}{qa}
\eqq{Q_{1,2}(t_{1},t_{2})= t^{1-\alpha}_{2} \iztjd \tamjd \tau^{\alpha-1} \izta \tap p^{\alpha-1} dp \dt,}{qd}
then using (\ref{eg}), (\ref{jb}),  (\ref{pa}) and (\ref{zalA}) we notice,
\[
|W_{1,1}(t_{1},t_{2})|  \leq c_{0} \ccT Q_{1,1}(t_{1},t_{2}),
\]
\[
|W_{1,2}(t_{1},t_{2})|  \leq c_{0} \ccT Q_{1,2}(t_{1},t_{2}).
\]
Now, the estimate below follows from Lemma~\ref{estiQjj},
\[
\lim_{t_{2}\rightarrow t_{1}}|W_{1}(t_{1},t_{2})|\leq \lim_{t_{2}\rightarrow t_{1}}c_{0}\ccT (Q_{1,1}(t_{1},t_{2})+Q_{1,2}(t_{1},t_{2}))=0.
\]

Next, we shall deal with  $W_{2}$. 
While keeping in mind the positivity of $t_{1}$ and  (\ref{of}), we perform differentiation with respect to $\tau$. We obtain,
\begin{eqnarray*}
 &
W_{2}(t_{1},t_{2})&=
 t^{1-\alpha}_{2} \iztj [\tamjd -\tamjj] \left[\izta   \tap B (p,\tau)c'(p) dp \right]_{\tau} \dt
\\
&&=-\alpha t^{1-\alpha}_{2} \iztj [\tamjd -\tamjj] \izta   \tapj B (p,\tau) c'(p)dp  \dt
\\
&&\quad+t^{1-\alpha}_{2} \iztj [\tamjd -\tamjj] \izta   \tap (B (p,\tau))_{\tau} c'(p)dp  \dt \\
&&\equiv  -\alpha W_{2,1}(t_{1},t_{2})+W_{2,2}(t_{1},t_{2}).
\end{eqnarray*}
Similarly to the estimates of $W_1$, we introduce new functions,
\eqq{Q_{2,1}(t_{1},t_{2})=t^{1-\alpha}_{2} \iztj [\tamjj -\tamjd] \izta   \tapj p^{\alpha-1} [s(\tau)-s(p)]dp  \dt,}{qf}
\eqq{Q_{2,2}(t_{1},t_{2})= t^{1-\alpha}_{2} \iztj [\tamjj -\tamjd] \tau^{\alpha-1} \izta   \tap p^{\alpha-1} dp   \dt .}{qg}
From (\ref{eg}), (\ref{egt}), (\ref{pa}) and (\ref{zalA}) we deduce that
\[
|W_{2,1}(t_{1},t_{2})| \leq c_{0} \ccT Q_{2,1}(t_{1},t_{2}),
\]
\[
|W_{2,2}(t_{1},t_{2})| \leq c_{0} \ccT Q_{2,2}(t_{1},t_{2}).
\]
Thus,
Lemma~\ref{estiQdd} implies
\[
\lim_{t_{2}\rightarrow t_{1}^{+}}W_{2}(t_{1},t_{2})=0.
\]
Obviously, we also have
\[
\lim_{t_{2}\rightarrow t_{1}^{+}}W_{3}(t_{1},t_{2})=0,
\]
and in this way we proved (\ref{qh}).

\paragraph{Case of $\cP_{2}$.} Using (\ref{zalA}) and the continuity of $c$  we notice that
\[
\cP_{2}(t)= t^{1-\alpha} \frac{\dos (t)}{s(t)} \df c(t).
\]
Hence, $\cP_{2}\in C([0,T];\rr^{m})$.

\paragraph{Case of $\cP_{3}$.} For all
$c\in X(T)$  we notice an obvious bound,
\eqq{|c(p)|\leq \ccT .  }{pad}
\no We shall show that
\eqq{
\cP_3\in C([0,T];\rr^{m}). }{qhd}
Similarly to the case of $\cP_1$, if we keep in mind Lemma~\ref{lemone} and Lemma~\ref{lemtwo}, then it is enough to prove  that
\eqq{t\mapsto t^{1-\alpha} \izt \tamj \left[ \izta \tap \df (p,\tau)  c(p)dp \right]_{\tau} \dt=:\cP_{3,1} \in C([0,T];\rr^{m}).}{pohd}
We are going to show that the following difference converges to zero when $t_{2}\rightarrow t_{1}$.
\begin{eqnarray*}
&&t^{1-\alpha}_{2} \iztd \tamjd \left[ \izta \tap \df (p,\tau) c(p)dp \right]_{\tau} \dt\\
&&- t^{1-\alpha}_{1} \iztj \tamjj \left[ \izta \tap \df (p,\tau) c(p)dp \right]_{\tau} \dt
\\
&=&t^{1-\alpha}_{2} \iztjd \tamjd \left[ \izta \tap \df (p,\tau) c(p) dp \right]_{\tau} \dt
\\
&&+t^{1-\alpha}_{2} \iztj [\tamjd -\tamjj] \left[ \izta \tap \df (p,\tau) c(p) dp \right]_{\tau} \dt
\\
&&+(t^{1-\alpha}_{2} - t^{1-\alpha}_{1})\iztj \tamjj \left[ \izta \tap \df (p,\tau) c(p)  dp \right]_{\tau} \dt \\
&\equiv& W_{4}(t_{1},t_{2})+W_{5}(t_{1},t_{2})+W_{6}(t_{1},t_{2}).
\end{eqnarray*}
If $t_{1}=0$, then we have $W_{5}(0,t_{2})=0$ and $W_{6}(0,t_{2})=0$,  as a result $\cP_{3,1}$
is continuous at zero if and only if $ \lim_{t_{2}\rightarrow 0 }W_{4}(0,t_{2})=0$.
This is why we estimate $W_{4}$. We proceed as in the case of $W_{1}$ and $W_{2}$. We use (\ref{og}) to perform the differentiation with respect to $\tau$. Here is the result, 
\begin{eqnarray*}
&W_{4}(t_{1},t_{2})&=t^{1-\alpha}_{2} \iztjd \tamjd  \left[\izta   \tap \df (p,\tau) c(p) dp \right]_{\tau} \dt
\\
&&=-\alpha t^{1-\alpha}_{2} \iztjd \tamjd  \izta \tapj \df (p,\tau) c(p)dp \dt
\\
&&\quad+t^{1-\alpha}_{2} \iztjd \tamjd  \izta \tap \df_{\tau} (p,\tau) c(p) dp \dt \\
& &\equiv- \alpha W_{4,1}(t_{1},t_{2})+ W_{4,2}(t_{1},t_{2}).
\end{eqnarray*}
From (\ref{ja}), (\ref{jb}),  (\ref{pad}) and (\ref{zalA}) while keeping in mind the definitions of $ Q_{1,1}$ and $Q_{1,2}$, we obtain that
\[
|W_{4,1}(t_{1},t_{2})|  \leq c_{0} \ccT  Q_{1,1}(t_{1},t_{2}),
\]
\[
|W_{4,2}(t_{1},t_{2})|  \leq c_{0} \ccT Q_{1,2}(t_{1},t_{2}).
\]
Invoking (\ref{qb}) and (\ref{qe}) from Lemma~\ref{estiQjj} we get the following estimate
\[
\lim_{t_{2}\rightarrow t_{1}}|W_{4}(t_{1},t_{2})|\leq \lim_{t_{2}\rightarrow t_{1}}c_{0} \ccT(Q_{1,1}(t_{1},t_{2})+Q_{1,2}(t_{1},t_{2}))=0.
\]
Since $t_1$ may be zero we see that $\cP_{3,1}$ is continuous at $t=0$.

Now, we turn our attention to $W_{5}$. In this case we have $t_{1}>0$. We use (\ref{og}) to perform the differentiation with respect to $\tau$, resulting with
\begin{eqnarray*}
&W_{5}(t_{1},t_{2})&=
 t^{1-\alpha}_{2} \iztj [\tamjd -\tamjj]\left[ \izta   \tap \df (p,\tau)c(p) dp \right]_{\tau} \dt
\\
&&=-\alpha t^{1-\alpha}_{2} \iztj [\tamjd -\tamjj] \izta   \tapj \df (p,\tau) c(p)dp  \dt
\\
&&\quad
+t^{1-\alpha}_{2} \iztj [\tamjd -\tamjj] \izta   \tap (\df (p,\tau))_{\tau} c(p)dp  \dt \\
&&\equiv  -\alpha W_{5,1}(t_{1},t_{2})+W_{5,2}(t_{1},t_{2}).
\end{eqnarray*}
Making use of (\ref{ja}), (\ref{jb}), (\ref{pad}) and (\ref{zalA}) we arrive at the following estimates,
\[
|W_{5,1}(t_{1},t_{2})| \leq c_{0} \ccT Q_{2,1}(t_{1},t_{2}),
\]
\[
|W_{5,2}(t_{1},t_{2})| \leq c_{0} \ccT Q_{2,2}(t_{1},t_{2}).
\]
We finish estimating $W_5$ by invoking
Lemma~\ref{estiQdd} to obtain,
\[
\lim_{t_{2}\rightarrow t_{1}^{+}}W_{5}(t_{1},t_{2})=0.
\]
Since it is easy to see that
\[
\lim_{t_{2}\rightarrow t_{1}^{+}}W_{6}(t_{1},t_{2})=0,
\]
then we conclude the proof of (\ref{qhd}).

\paragraph{Case of $\cP_{4}$.} Since functions $s$ and $c$ are absolutely continuous, then due to (\ref{def_lambdy}) and the definition of $Ec$ we deduce that   $Ec\in AC$ too.  Since $P_4(t) = I^\alpha(Ec)$, then  by Lemma~\ref{lemone} we deduce that
\[
\cP_4\equiv t^{1-\alpha} (P_{4}c)'(t) = c_{0} t^{1-\alpha } \izt \tamj (E(\tau) c(\tau))' \dt + c_{0} E(0)c(0).
\]
We infer from (\ref{oi}) and (\ref{zalA}) that function $t\mapsto t^{1-\alpha}(Ec)'$ is essentially bounded, hence we may apply Lemma~\ref{lemthree} with $f=Ec$ to deduce that the $\cP_4$ is H\"older continuous on $[0,T]$, i.e. it has better regularity than required.

We also have to check continuity of $t\mapsto t^{1-\alpha}(\ia G^{\ep})'(t)$, where
function $G^{\ep}$ is smooth. This follows from Lemma~\ref{lemthree} applied to $G^{\ep}$. Thus, we finished the proof of (\ref{qj}).

The final step in the proof of $Pc\in X(T)$ is checking that $(Pc)(0) =c_m(0).$  For this purpose
we  shall check that $P_i(0) =0$, $i=1,\ldots,4$ as well as $I^\alpha G^\epsilon(0) =0$.

We notice that (\ref{eg}) and (\ref{ib}) imply
\begin{eqnarray*}
&|(P_{1}c)(t)| &\leq c_{0} \ccT \izt \tamj \izta \tap [s(\tau) -s(p)]p^{\alpha-1} dp\dt
\\
&&\leq c_{0} \ccT \izt \tamj \izta p^{\alpha-1} dp\dt \\
&&=c_{0} \ccT t^{2\alpha} \underset{t \rightarrow 0}{ \longrightarrow 0 }.
\end{eqnarray*}
In next step we see that we deduce from (\ref{zalA}) that
\[
|(P_{2}c)(t)|\leq c_{0} \ccT t^{\alpha} \underset{t \rightarrow 0}{ \longrightarrow 0 }.
\]
Estimates (\ref{ja})  and (\ref{ib}) permit us to bound $P_{3}$ as follows,
\begin{eqnarray*}
&|(P_{3}c)(t)| &\leq c_{0} \ccT \izt \tamj \izta \tap [s(\tau) -s(p)]\dos(p) dp\dt
\\
&&\leq c_{0} \ccT \izt \tamj \izta \dos(p)dp \dt \\
&&\leq c_{0} \ccT  t^{\alpha} [s(t)-s(0)] \underset{t \rightarrow 0}{ \longrightarrow 0 }.
\end{eqnarray*}

The remaining terms, $P_4 =I^\alpha(Ec)$ and $I^\alpha G^\epsilon$ tend to zero as $t$ goes to $0$, because they are fractional integrals of bounded functions. Hence, we conclude
that $(Pc)(0)=c_{m}(0)$ and Lemma~\ref{defX} is proved.
\end{proof}

\subsection{Operator $P$ is a contraction}
Here, we show that we can choose $T$ sufficiently small to make $P$ a contraction on $X(T)$.
\begin{lemma}\
There exists  $T_{0}\le T$, which depends only on $b,m,\alpha$, $T$  and $\| {s}\|_{X(T)}$ such that if  $T_{1}<T_{0}$, then  $P: X(T_{1})\rightarrow X(T_{1})  $ is contraction.
\label{kont}
\end{lemma}

\begin{proof}
We take a number $T_1\le T$, any elements $c_{1},c_{2} \in \xtj$ and their difference $\bc=c_{1}-c_{2}$. Of course, $\bc(0)=0$
and we have,
\eqq{|\bc'(p)| \leq \nbcj p^{\alpha-1}, }{pe}
\eqq{|\bc(p)| \leq \nbcj p^{\alpha}. }{pf}
Since operators $P_{i}$ are linear, it is sufficient to show that
\eqq{\| P_{i}\bc \|_{\xtj} \leq \rho_i(T_1) \| \bc \|_{\xtj}, }{pd} 
where $\rho_i(T_1)$ go to zero, when $T_1\to0$, $i=1,\ldots,4$. We recall that we denote by $c_{0}$ a generic constant depending  only on the data specified in the assumptions. We estimate operators $P_{i}$ one by one.

\paragraph{Case of  $P_{1}$.} Making use of (\ref{pe}), (\ref{eg}) and (\ref{ib}), we can estimate $P_{1}\bc$ in the following way, for $t\le T_1$,
\[
|P_{1} \bc(t) | \leq c_{0} \nbcj \izt \tamj \izta p^{\alpha-1} dp \dt= c_{0} \nbcj t^{2\alpha}.
\]
As a result we see,
\eqq{\| P_{1} \bc \|_{C(0,T_{1})} \leq c_{0} T^{\alpha}_{1}\nbcj.}{ie}
After an application of Lemma~\ref{lemone} and Lemma~\ref{lemtwod},  then making  use of estimates (\ref{eg}), (\ref{egt}), (\ref{pe}) and  assumption (\ref{zalA}) we arrive at
\[
|t^{1-\alpha} (P_{1}\bc)'(t)|\leq c_{0} \nbc [Q_{1,1}(0,t)+Q_{1,2}(0,t)].
\]
The estimates on $Q_{1,1}$ and $Q_{1,2}$ provided by Lemma \ref{estiQjj}, combined with
(\ref{ib}) give us
\eqq{
|t^{1-\alpha} (P_{1}\bc)'(t)|\leq c_{0}  T^{\alpha}_{1}\nbcj \hd \mbox{ for } \hd t\in [0,T_{1}].
}{qk}

\paragraph{Case of  $P_{2}$.} Taking into account (\ref{pf}) and (\ref{zalA}), we obtain the following estimate,
\[
|(P_{2}c)(t)| \leq c_{0} \nbcj [s(t)-s(0)].
\]
Hence, (\ref{ib}) implies that
\eqq{\| P_{2} \bc \|_{C(0,T_{1})} \leq c_{0} T^{\alpha}_{1}\nbcj.}{ql}
Finally, combing (\ref{pf}) with assumption (\ref{zalA}) gives us
\eqq{
|t^{1-\alpha} (P_{2}\bc)'(t)|\leq c_{0} T^{\alpha}_{1}\nbcj\qquad\hbox{for all} \quad t\le T_1.
}{qm}

\paragraph{Case of  $P_{3}$.} In order to estimate $P_{3}$, we use assumption (\ref{zalA})) and the bounds
(\ref{pf}), (\ref{ja}),  (\ref{ib}). Then, we clearly have
\[
|P_{3} \bc(t) | \leq c_{0} \nbcj \izt \tamj \izta p^{2\alpha-1} dp \dt= c_{0} \nbcj t^{3\alpha},\qquad\hbox{for all} \quad t\le T_1.
\]
Hence,
\eqq{\| P_{3} \bc \|_{C(0,T_{1})} \leq c_{0} T^{\alpha}_{1}\nbcj.}{qn}
Analogously to the case of $P_{1}$, using Lemma~\ref{lemone} and Lemma~\ref{lemtwod} and next (\ref{ja}), (\ref{jb}), (\ref{pf}) and (\ref{zalA}), we obtain
\[
|t^{1-\alpha} (P_{3}\bc)'(t)|\leq c_{0} t^{\alpha}\nbc [Q_{1,1}(0,t)+Q_{1,2}(0,t)],
\]
where $Q_{1,i}$ were defined in (\ref{qa}) and (\ref{qd}).

Hence from (\ref{qc}), (\ref{qe}) and (\ref{ib}), we get
\eqq{
|t^{1-\alpha} (P_{3}\bc)'(t)|\leq c_{0}  T^{\alpha}_{1}\nbcj \hd \mbox{ for } \hd t\in [0,T_{1}].
}{qo}

\paragraph{Case of  $P_{4}$.} Using (\ref{pf}), we have
\[
|(P_{4}c)(t)|\leq c_{0} \izt \tamj E(\tau) \bc (\tau) \dt \leq c_{0} \nbcj t^{\alpha},
\]
hence
\eqq{\| P_{4} \bc \|_{C(0,T_{1})} \leq c_{0} T^{\alpha}_{1}\nbcj.}{qp}
It remains to show the estimate for $t^{1-\alpha} (P_{4}\bc)'(t)$. To that end we will use
Lemma~\ref{lemone} and the fact that $\bc(0)=0$.
Then, from (\ref{oi}), (\ref{pe}) and (\ref{zalA}), we get
\[
|t^{1-\alpha} (P_{4}\bc)'(t)|\leq c_{0} t^{1-\alpha} \izt \tamj [E(\tau)\bc (\tau)]_{\tau} \dt
\]
\[
 \leq c_{0} \nbcj  t^{1-\alpha} \izt \tamj \tau^{\alpha-1} \dt  =c_{0}   t^{\alpha}\nbcj,
\]
so
\eqq{
|t^{1-\alpha} (P_{4}\bc)'(t)|\leq c_{0}  T^{\alpha}_{1}\nbcj \hd \mbox{ for } \hd t\in [0,T_{1}].
}{qq}
Finally, from (\ref{ie})-(\ref{qq}), we get
\eqq{\| P\bc \|_{X(T_{1})} \leq c_{0} T^{\alpha}_{1} \nbcj.}{qr}
This finishes the proof.
\end{proof}

\subsection{Finding the solution on $[0,T]$}
We want to find the solution on the interval $[0,T]$, which is given to us. However, due to the non local character of the problem, we cannot use directly the approach applicable to ODE's, but we have to modify it.

\begin{lemma} Let us
suppose that the assumptions of Theorem~\ref{main} hold, in particular
$T>0$ is given. Then, there exists a unique function $c\in X(T)$, i.e. $c \in C^{1}((0,T];\rr^{m})$ and $t^{1-\alpha}c' \in C([0,T];\rr^{m})$, which is a
solution of (\ref{eh}).
\label{konti}
\end{lemma}

\begin{proof}
We know from Lemma \ref{defX} that for any $T>0$ operator $P$ maps $X(T)$ into itself. By the Banach fixed point theorem and Lemmas~\ref{defX}, \ref{kont} we obtained the unique solution of (\ref{eh}) for $t\in [0,T_{1}]$.
We want to  extend it for the whole interval $[0,T]$. For this purpose, we assume that solution $c_m$ is given on $[0,t_*]$ and $c_m(t_*)\in \rr^m$.
We shall see that we can extend this function on $[0,t_e]$, where $t_e>t_{*}$. The key observation will be that $t_e - t_{*}\ge\delta>0$, where $\delta = \delta(\alpha, m, T, \|s\|_{X(T)})$.

Let us suppose that for $t_*>0$ function $c_m\in X(t_*)$ is a unique solution to  (\ref{eh}). We introduce
$$
Y_{t_*}(T) =\{ u\in X(T) :\ u|_{[0,t_*]} = c_m\},
$$
we notice that for $c_1,$ $c_2\in X(T)$ such that  $c_1(t) = c_2(t)$ for $t \in [0,t_*],$ after writing $x_i = c_i|_{[t_*,T]}$, $i=1,2$, we have
$$
\| c_1-c_2 \|_{X(T)} \le \max\{1,T^{1-\alpha}\} \| x_1-x_2 \|_{C^1([t_*,T])}
$$
and
$$
\| x_1-x_2 \|_{C^1([t_*,T])}  \le \max\{1,t_*^{\alpha-1}\} \| c_1-c_2 \|_{X(T)}.
$$
As a result, the following metrics on $Y_{t_*}(T)$
$$
\rho_1(c_1,c_2) =  \| c_1-c_2 \|_{X(T)}\qquad\hbox{and}\qquad\rho_2(c_1,c_2) =  \| c_1-c_2 \|_{C^1([t_*,T])}
$$
are equivalent. It is of our advantage to rewrite the fixed point problem (\ref{r:ca}) in $X(T)$,
\begin{equation}\label{r:cap}
 c_m(t) =  c_m(0) + (Pc_m)(t)
\end{equation}
as a fixed point problem in $Y_{t_*}(T)$. For the sake of simplicity of notation we write $y(t) := c_m(t)$ for $t\in[0,t_*]$ and $x:= c|_{[t_*,T]}$ for $c\in Y_{t_*}(T)$. For $t>t_*$ we can rewrite (\ref{r:cap}) as
\begin{eqnarray*}
&x(t) = c(t) &= c_m(0) + (P(x+y))(t) \\
&& = (P^{t_*}x)(t) + (\bar P y)(t),
\end{eqnarray*}
where $\bar P y$ is fixed, because it depends upon history until $t=t_*$ and
operator $P^{t_*}: Y_{t_*}(T)\to Y_{t_*}(T)$ is given as,
$$
P^{t_*} = P^{t_*}_1 + P^{t_*}_2 + P^{t_*}_3 + P^{t_*}_4.
$$
Here are the definitions,
\begin{eqnarray*}
 (\bar P y)(t)&=& c(0) -\ca \int_{0}^{t_*} \tamj \izta\tap B(p, \tau )y'(p) dp \dt - \ca \int_{0}^{t_*} \frac{\dos (p)}{s(p)} \hat{D} y(p) dp
\\
 &&-\ca\int_{0}^{t_*} \tamj \izta \tap D(p,\tau) y(p ) dp \dt -  \ca\int_{0}^{t_*} \tamj   E(\tau) y(\tau)\dt\\
\\
&&-\ca \int_{t_*}^{t} \tamj \int_{0}^{t_*}\tap B(p, \tau )y'(p) dp \dt \\
&&- \ca\int_{t_*}^{t} \tamj \int_{0}^{t_*} \tap D(p,\tau) y(p ) dp \dt+ \jgja \ja G^{\ep}(t),
\end{eqnarray*}
and
\begin{eqnarray*}
 ({P^{t_*}_{1}}x)(t)&=&-\ca \itt \tamj \itta\tap B(p, \tau )x'(p) dp \dt,\\
 ({P^{t_*}_{3}}x)(t)&=&-\ca\itt \tamj \itta \tap D(p,\tau) x(p ) dp \dt,
\end{eqnarray*}
$$
({P^{t_*}_{2}}x)(t)=-\ca\itt \frac{\dos (p)}{s(p)} \hat{D} x(p) dp  \dt,\qquad
({P^{t_*}_{4}}x)(t)= -  \ca\itt \tamj   E(\tau) x(\tau)\dt.
$$
We shall show that operator  $P^{t_*}$ is a contraction on $Y_{t_*}(T)$, hence it has a unique fixed point, provided $T-t_*$ is small enough. Moreover, we have to prove that the difference $T-t_*$ can be estimated by a universal constant.

%

We will estimate the Lipschitz constant of each $P^{t_*}_i$, $i=1,\ldots,4$ on $Y_{t_*}(T)$ separately. For $c^{1}, c^{2} \in Y_{t_*}(T)$, we shall write $x = c^{1}-c^{2}$, moreover $\| x \| $ denotes $\| x \|_{C^{1}([t_*,T];\rr^{m})}$. We also notice that the definition of $Y_{t_*}(T)$ implies
\eqq{x (t_*)=0, \hd \hd |x (t) | \leq \| x \| (t-t_*),}{char}

\paragraph{Case of $P^{t_*}_{1}$. } Using an argument similar to that in the proof of Lemma~\ref{lemtwod}, we deduce that function $g_{3}(\tau) = \itta \tap B(p, \tau) x'(p) dp$ is absolutely continuous. Hence by Lemma~\ref{lemone}  and (\ref{of}) we get
\begin{eqnarray*}
(P^{t_*}_{1}x )'(t)& = &\alpha \ca \itt \tamj \itta \tapj B(p, \tau) x '(p)dp \dt \\
&& -\ca \itt \tamj \itta \tap B_{\tau}(p, \tau) x '(p)dp \dt.
\end{eqnarray*}
In order to estimate the derivative of $P^{t_*}_{1} x$ we need to use (\ref{eg}), (\ref{od}), (\ref{egt}) and (\ref{zalA}). As a result, we obtain for $t_*\geq T_{1}$,
\begin{eqnarray*}
 |(P^{t_*}_{1} x )'(t)| &\leq& c_{0} \| x \|  \itt \tamj \itta \tap  p^{\alpha-1} dp \dt \\&&+ c_{0} \| x \|  \itt \tamj \tau^{\alpha-1} \itta \tap dp \dt
\\
&\leq & c_{0} \| x \| \itt \tamj \dt + c_{0} \| x \| \itt \tamj \tau^{\alpha-1} (\tau - t_*)^{1- \alpha} \dt
\\
&\leq &c_{0} \| x \| (t -t_* )^{\alpha}+c_{0} \| x \| T_{1}^{\alpha-1} (t -t_* ),
\end{eqnarray*}
where $T_{1}$ is given in Lemma~\ref{kont}. Since $(P^{t_*}_{1} x )(t_*)=0$ we arrive at
\eqq{ \| P^{t_*}_{1} x \| \leq c_{0} \| x \| (T-t_*)^{\alpha}. }{ra}

\paragraph{Case of $P^{t_*}_{2}$. } Here, the estimates are simpler. From (\ref{char}) we get
\[
|(P^{t_*}_{2}x)'(t)|= c_{0} \frac{\dos(t)}{s(t)}|\hat{D} x (t)|
\leq c_{0} T^{\alpha-1}_{1} |x (t)| \leq c_{0} T_{1}^{\alpha-1} \| x \| (t- t_*).
\]
Thus,
\eqq{ \| \pf_{2}x \| \leq c_{0} \| x \| (T-t_*)^{\alpha}. }{rb}

\paragraph{Case of $P^{t_*}_{3}$. } Arguing as in the proof of Lemma~\ref{lemtwo}, we deduce that function $g_{4}(\tau) = \itta \tap \df(p, \tau)x(p) dp$ is absolutely continuous. Then,
\begin{eqnarray*}
 (P^{t_*}_{3}x )'(t) &=& \alpha \ca \itt \tamj \itta \tapj \df(p, \tau)x (p)dp \dt
\\
&&-\ca \itt \tamj \itta \tap \df_{\tau}(p, \tau)x (p)dp \dt,
\end{eqnarray*}
where we used Lemma~\ref{lemone}  and (\ref{og}).
Combining the estimates (\ref{ja}), (\ref{od}),(\ref{jb}), (\ref{zalA}) and (\ref{char}), we arrive at
\begin{eqnarray*}
|(P^{t_*}_{3} x )'(t)| &\leq &c_{0}  \itt \tamj \itta \tap \dos (p) p^{\alpha-1} | x  (p)| dp \dt
\\
&&+c_{0}  \itt \tamj \itta \tap \dos (p) \dos(\tau) | x  (p)| dp \dt
\\
&\leq& c_{0} \| x \| T_{1}^{\alpha-1} \itt \tamj \itta \tap  p^{\alpha-1} (p-t_*) dp \dt
\\
&\leq& c_{0} \| x \| T_{1}^{\alpha-1} (t- t_*)^{1+\alpha}.
\end{eqnarray*}
Hence,
\eqq{ \| P^{t_*}_{3}  x  \| \leq c_{0} \| x \| (T-t_*)^{\alpha}. }{rc}

\paragraph{Case of $P^{t_*}_{4}$. } Using an analog of Lemma~\ref{lemone} and $ x  (t_*) =0$, we get
\[
|(P^{t_*}_{4}  x )'(t)| \leq \itt \tamj |E'(\tau)| | x  (\tau) |\dt + \itt \tamj |E(\tau)| | x ' (\tau) |\dt .
\]
Hence,
\begin{eqnarray*}
 |(P^{t_*}_{4}  x )'(t)| &\leq& c_{0} \| x \| \itt \tamj \tau^{\alpha-1} (\tau - t_* )\dt + c_{0} \| x \| (t-t_*)^{\alpha}
\\
&\leq &c_{0} \| x \| T^{\alpha-1}_{1} (t -t_*)^{1+\alpha} + c_{0} \| x \| (t-t_*)^{\alpha},
\end{eqnarray*}
where we applied (\ref{oi}), (\ref{char}) and (\ref{zalA}).
Thus,
\eqq{ \| P^{t_*}_{4}  x  \| \leq c_{0} \| x \| (T-t_*)^{\alpha}. }{rd}
Finally, from (\ref{ra})-(\ref{rd}), we have
\[
\| P^{t_*} x  \| \leq c_{0} \|x\| (T-t_*)^{\alpha},
\]
and the proof is finished.
\end{proof}

\noindent{\it Proof of Theorem~\ref{przyb}. }
Lemma~\ref{konti} yields $c_{m} \in AC([0,T];\rr^{m})$, a fixed point of $P$ in $X(T)$, where $T>0$ was given in advance, i.e. (\ref{eh}) holds. We notice that $c_{m}$ satisfies the assumptions of Lemma~\ref{lemthree}, thus $\da c_{m} \in C^{0,1-\alpha}_{loc}((0,T];\rr^{m})$.
Thus, if we apply Caputo derivative $\da$ to the both sides of (\ref{eh}), we obtain that $v_{m}$ defined by (\ref{ec}) satisfies the system  (\ref{ef}), (\ref{eff}). Finally, we have $v_{m,t}=\snm c_{n,m}'(t) \vn(x,t)+\snm \cnmt \varphi_{n,t}(x,t) $. By Lemma~\ref{konti} we obtain the continuity of $t^{1-\alpha}c_{m}'$. On the other hand,  (\ref{zalA}) implies the continuity of  $t^{1-\alpha} \varphi_{n,t}(x,t)$ in  $\overline{\qs}$, hence
$t^{1-\alpha}v_{m,t}$ is continuous on $\overline{\qs}$. \qed


\section{Existence of weak solutions}\label{s:w}

In order to deduce appropriate estimates for $v_{m}$, we need to state a technical lemma.

\begin{lemma}
We assume that $s$ satisfies (\ref{zalA}), $b$ is positive and $\alpha\in (0,1)$. Let us suppose that $u$ is continuous on $\overline{\qs} $ and such that $t^{1-\alpha}u_{t} \in C (\overline{\qs})$. Then, for each   $t\in(0,T]$ the following estimate holds,
\[
D^{\alpha} \| u(\cdot,t)\|^{2}_{L^{2}(0,s(t))} +
\frac{t^{-\alpha}}{\Gamma(1-\alpha)} \int_{0}^{b} |u(x,t)-u(x,0)|^{2}dx
\]
\[
 +
\jgja\int_{b}^{s(t)} [t-\sj]^{-\alpha} |u(x,t)-u(x,\sj)|^{2}dx
\]
\eqq{ \leq 2\int_{0}^{s(t)} D^{\alpha}_{s}u(x,t) \cdot u(x,t)dx +\jgja\izt \ta |u(s(\tau),\tau)|^{2}\dot{s}(\tau) \dt.   }{ba}
\label{stwjeden}
\end{lemma}

\begin{proof}
We recall that $\das $ was defined in (\ref{aa}).
We can deduce from our assumptions
that
 function $\xi(t) =\| u(\cdot , t)\|^{2}_{L^{2}(0, s(t))}$ is absolutely continuous on $[0,T]$, thus the first term on the left-hand-side of (\ref{ba}) is defined for a.e. $t\in (0,T)$. Furthermore,
$\xi$ satisfies the assumptions of Lemma~\ref{lemthree}, thus $\das \xi \in C^{0,1-\alpha}_{loc}((0,T])$, so both functions are defined for each $t\in (0,T]$.  This is why we can write
\begin{eqnarray*}
&&\gja \int_{0}^{s(t)} D^{\alpha}_{s}u(x,t) \cdot u(x,t)dx +\frac{1}{2} \izt \ta |u(s(\tau),\tau)|^{2}\dot{s}(\tau) \dt\\
&&- \frac{\gja}{2} D^{\alpha} \| u(\cdot,t)\|^{2}_{L^{2}(0,s(t))}
\\
&= &\izt\izst \ta \utt  u(x,t)dx \dt+\frac{1}{2} \izt \ta |u(s(\tau),\tau)|^{2}\dot{s}(\tau) \dt
\\
&&- \frac{1}{2}\izt \ta \frac{d}{d \tau} \izst |\uxt|^{2}dx \dt
\\
&=&\izt\izst \ta \utt  u(x,t)dx \dt
- \izt \ta  \izst \uxt u_{\tau}(x, \tau)  dx \dt
\\
&= &\izt \izst \ta \utt [u(x,t)- \uxt]dx \dt
\\
&=&\left(  \ith \izst + \itht \izst\right)\ta \utt [u(x,t)- \uxt]dx \dt \\
&\equiv &I^{1}_{h}+I^{2}_{h},
\end{eqnarray*}
where $h>0$. Careful integration by parts gives us,
\begin{eqnarray*}
I^{1}_{h} &=&-\frac{1}{2} \ith \izst  \ta  \left( |u(x,t)- \uxt|^{2} \right)_{\tau}dx \dt
\\
&=&-\frac{1}{2} \left( \ibh +\ibsh \right) \ta  \left( |u(x,t)- \uxt|^{2} \right)_{\tau}\dt dx
\\
&=&\frac{\alpha}{2} \left( \ibh +\ibsh \right) \taj   |u(x,t)- \uxt|^{2} \dt dx
\\&&
\left.-\frac{1}{2}  \inb  \ta   |u(x,t)- \uxt|^{2} dx \right|_{\tau =0}^{\tau = t-h}\\&&
\left.-\frac{1}{2} \ibth  \ta   |u(x,t)- \uxt|^{2} dx \right|_{\tau =\sj}^{\tau = t-h}
\\
&=&\frac{\alpha}{2} \ith \izst \taj   |u(x,t)- \uxt|^{2} dx\dt
\\
&&+\frac{1}{2} \inb t^{-\alpha} |u(x,t)-u(x,0)|^{2}dx +\frac{1}{2} \ibth (t-\sj)^{-\alpha} |u(x,t)-u(x, \sj)|^{2}dx
\\
&&-\frac{1}{2} h^{-\alpha} \int_{0}^{s(t-h)}    |u(x,t)- u(x,t-h)|^{2} dx.
\end{eqnarray*}
We denote the last term by $I^{3}_{h}$. Therefore, the proof of (\ref{ba}) will be finished if we show that $I^{2}_{h}$ and $I^{3}_{h}$ have limits equal to zero, when  $h\rightarrow 0$. In order to prove this we assume that $h\leq \frac{1}{2} t$. Then, from (\ref{zalA}) we obtain
\begin{eqnarray*}
|I^{2}_{h}| &\leq & 2 \| u \|_{C(\qs)} \| t^{1-\alpha}u_{t} \|_{C(\qs)} \int_{t-h}^{t} \int_{0}^{s(\tau)} \ta \tau^{\alpha-1} dx \dt
\\
&=& 2 s(T)\| u \|_{C(\qs)} \| t^{1-\alpha}u_{t} \|_{C(\qs)} \int_{t-h}^{t}  \ta \tau^{\alpha-1}  \dt
\\
&\leq &2 s(T)\| u \|_{C(\qs)} \| t^{1-\alpha}u_{t} \|_{C(\qs)} \left(\frac{t}{2}\right)^{\alpha-1}  \frac{h^{1-\alpha}}{1-\alpha} \underset{h\rightarrow 0}{\longrightarrow}0.
\end{eqnarray*}
Similarly, we get
\[
|I^{3}_{h}| \leq h^{2-\alpha} (t-h)^{2\alpha-2}  \| t^{1-\alpha}u_{t} \|^{2}_{C(\qs)} s(t-h)
\leq h^{2-\alpha} \left(  \frac{t}{2} \right)^{2\alpha-2}  \| t^{1-\alpha}u_{t} \|^{2}_{C(\qs)} s(t)\underset{h\rightarrow 0}{\longrightarrow}0,
\]
hence the proof is finished.
\end{proof}

\begin{lemma}\label{oszac}
Let us suppose that the assumptions of Theorem~\ref{main} are satisfied. Then, for each $m\in \N$, 
and $t\in (0,T]$ the  approximate solutions $v_{m}$, given in Theorem~\ref{przyb}, satisfy the following estimate
\begin{eqnarray}\label{wc}
 &&I^{1-\alpha} \| v_{m}(\cdot,t)\|^{2}_{L^{2}(0,s(t))} +\jgja\int_{0}^{t} \tau^{-\alpha}\int_{0}^{b} |v_{m}(x,\tau)-v_{m}(x,0)|^{2}dx \dt \nonumber
\\
&& +\jgja \int_{0}^{t}\int_{b}^{s(\tau)} [\tau-\sj]^{-\alpha} |v_{m}(x,\tau)|^{2}dx\dt+ \int_{0}^{t} \izst |v_{m,x}(x,\tau)|^{2}  dx \dt\nonumber
\\
&\leq & \frac{s(T)^{2}}{2} \|g\|^{2}_{L^{2}(\qst)} +\frac{t^{1-\alpha}}{\ga (1-\alpha)} \| v_{0}\|_{L^{2}(0,b)}^{2} + \delta(\ep),
\end{eqnarray}
where  $\lim\limits_{\ep\rightarrow 0 } \delta(\ep)=0$.
\end{lemma}

\begin{proof}
We multiply  (\ref{ef}) by $c_{k,m}(t)$ and sum it up over $k$ from $0$ to $m$
\[
\izs \das v_{m}(x,t) \cdot v_{m}(x,t) dx + \izs |v_{m,x}(x,t)|^{2}  dx=\izs \gem(x,t) v_{m}(x,t) dx.
\]
Applying Theorem~\ref{przyb}, we deduce that $v_{m}$ satisfies the hypothesis of Lemma~\ref{stwjeden}. Hence, from this Lemma and the fact that $v_{m}(s(t),t)=0$ for $t\in [0,T]$, we get
\[
\frac{1}{2} D^{\alpha} \| v_{m}(\cdot,t)\|^{2}_{L^{2}(0,s(t))} +\jgja\frac{t^{-\alpha}}{2} \int_{0}^{b} |v_{m}(x,t)-v_{m}(x,0)|^{2}dx
\]
\[
+ \jgja\frac{1}{2} \int_{b}^{s(t)} [t-\sj]^{-\alpha} |v_{m}(x,t)|^{2}dx+ \izs |v_{m,x}(x,t)|^{2}  dx
\]
\eqq{
\leq \izs \gem(x,t) v_{m}(x,t) dx.
}{wb}
Using Young's 
inequality, we have
\begin{eqnarray*}
 &&\izt \izst \gem(x,\tau) v_{m}(x,\tau) dx \dt \\
&\leq &\frac{s(T)^{2}}{4} \izt \izst |\gem (x,\tau)|^{2}dx \dt + s(T)^{-2} \izt \izst |v_{m}(x,\tau)|^{2}dx \dt
\\
&\leq& \frac{s(T)^{2}}{4} \int^{t+\ep}_{0}\izst |g (x,\tau)|^{2}dx \dt + \frac{1}{2} \izt \izst |v_{m,x}(x,\tau)|^{2}dx \dt,
\\
&\leq & \frac{s(T)^{2}}{4} \| g \|^{2}_{L^{2}(\qst)}+ \frac{1}{2} \izt \izst |v_{m,x}(x,\tau)|^{2}dx \dt + \delta(\ep),
\end{eqnarray*}
where $\delta (\ep) = \frac{s(T)^{2}}{4} \sup_{t>0} \int^{t+\ep}_{t}\izst |g (x,\tau)|^{2}dx \dt$. Due to the absolute continuity of the Lebesgue integral we notice that $\lim_{\ep\rightarrow 0 } \delta(\ep)=0$. Next, we integrate (\ref{wb}) with respect to $t$ while applying (\ref{r:zli}) and  (\ref{doda}) to the first term. We use the above inequality to estimate the right-hand-side.
\end{proof}

Finally, after these preparations, we are ready to prove the main result, Theorem~\ref{main}.

\begin{proof}[Proof of Theorem~\ref{main}]
By the assumption $h\in C^{1}([0,T]),$ so $\da h \in C([0,T])$, hence $u$ is a weak solution of (\ref{a}) if and only if $v$ is a weak solution of (\ref{hc}). Here, $g\in L^{2}(\qs)$ and $v_{0}\in L^{2}(0,b)$ were defined in (\ref{wd}). Therefore, we have to show that (\ref{hbb}) holds.

For $\ep=1/m$  Theorem~\ref{przyb} yields  $v_{m}$ satisfying (\ref{ef}). We multiply this equality by a test function $\psi\in C^{1}([0,T])$ such that  $\psi(T)=0$. Then, we get
\[
\int_{0}^{T}\izs \das v_{m}(x,t) \cdot \vk(x,t) dx \psi(t) dt + \int_{0}^{T} \izs v_{m,x}(x,t) \cdot \vkx(x,t) dx \psi(t) dt
\]
\eqq{
=\int_{0}^{T}\izs \gem(x,t) \vk(x,t) dx \psi(t) dt.
}{wh}
We note that $v_{m}$ is absolutely continuous with respect to $t$, so  is $I^{1-\alpha}v_{m}$. Thus, we may integrate by parts to get
\begin{eqnarray*}
&&\int_{0}^{T}\izs \das v_{m}(x,t) \cdot \vk(x,t) dx \psi(t) dt
\\
&=&\int_{0}^{b} \int_{0}^{T} \frac{d}{dt}I^{1-\alpha} [v_{m}(x,t) -v_{m}(x,0)] \cdot \vk(x,t)  \psi(t) dt dx
\\
&&+
\int_{b}^{s(T)} \int_{\sj}^{T} \frac{d}{dt}I^{1-\alpha} v_{m}(x,t)  \cdot \vk(x,t)  \psi(t) dt dx
\\
&=&-\int_{0}^{b} \int_{0}^{T} I^{1-\alpha} [v_{m}(x,t) -v_{m}(x,0)] \cdot \left[\vk(x,t)  \psi(t)\right]_{t} dt dx\\
&&-\int_{b}^{s(T)} \int_{\sj}^{T} I^{1-\alpha} v_{m}(x,t) \cdot \left[\vk(x,t)  \psi(t)\right]_{t} dt dx
\\
&=&-\int_{0}^{T} \izs I^{1-\alpha}_{s} [v_{m}(x,t) -\widetilde{v}_{m}(x,0)] \cdot \left[\vk(x,t)  \psi(t)\right]_{t} dx dt,
\end{eqnarray*}
where we used the fact that $v_{m}(x,\sj)=0.$\\
We shall show the following estimate
\eqq{\| I^{1-\alpha} \vtm ( \cdot, t ) \|_{L^{2}(0,s(T))} \leq I^{1-\alpha} \| v_{m} (\cdot, t ) \|^{2}_{L^{2}(0,s(t))} + \frac{1}{\Gamma(2-\alpha)} t^{1-\alpha}, \hd \mbox{ for  } t\in [0,T],  }{we}
where the tilde over $v_m$ denotes an extension by zero. Indeed, denoting by $\| \cdot \|$ the norm $\| \cdot  \|_{L^{2}(0,s(T))}$, we can write
\begin{eqnarray*}
\| I^{1-\alpha} \vtm ( \cdot, t ) \| &=& \sup_{\| \eta \|=1} \left| \int_{0}^{s(T)} I^{1-\alpha} \vtm ( x, t ) \eta(x) dx  \right|
\\
&\leq& \jgja \sup_{\| \eta \|=1}  \int_{0}^{t} \ta  \int_{0}^{s(T)}  |\vtm ( x, \tau ) \eta(x)| dx  \dt
\\
&\leq &\jgja   \int_{0}^{t} \ta   \|\vtm ( \cdot , \tau )\|   \dt
\le \jgja   \int_{0}^{t} \ta   [\|\vtm ( \cdot , \tau )\|^{2}+1]   \dt
\\
&=&I^{1-\alpha} \| \vtm (\cdot, t ) \|^{2} + \frac{1}{\Gamma(2-\alpha)} t^{1-\alpha},
\end{eqnarray*}
which leads to (\ref{we}).

By Lemma~\ref{oszac}, we  get the estimates of   $\{v_{m} \}, \{v_{m,x} \}$ in the norm $L^{2}(\qs)$ and  
by (\ref{we})  the bound for the  norm of $I^{1-\alpha} \widetilde{v}_{m}$ in  $L^{\infty}(0,T;L^{2}(0,s(T)) $.   As a result, there exists a subsequence (denoted again by $\{v_{m} \}$) and $q\in  L^{\infty}(0,T;L^{2}(0,s(T)))$ such that  $v_{m} \longrightarrow   v $, $v_{m,x}  \rightharpoonup v_{,x}$ in   $L^{2}(\qs)$ and  $I^{1-\alpha} \widetilde{v}_{m}   \overset{\ast}{\rightharpoonup} q$ in  $L^{\infty}(0,T;L^{2}(0,s(T))) $. Then  $\tilde{v}_{m} \longrightarrow \tilde{v} $  in   $L^{2}((0,s(T))\times (0,T))$. Thus, from inequality
\eqq{\| I^{1-\alpha} w \|_{L^{2}((0,s(T))\times (0,T))} \leq c_{0}\|  w \|_{L^{2}((0,s(T))\times (0,T))},}{wf}
we obtain  $I^{1-\alpha} \widetilde{v}_{m} \rightarrow  I^{1-\alpha} \widetilde{v}$ in  $L^{2}((0,s(T))\times (0,T))$. The uniqueness of the weak limit implies $I^{1-\alpha}\widetilde{v}=q\in L^{\infty}(0,T;L^{2}(0,s(T)))$.
From the estimate  $\left| \left[\vk(x,t)  \psi(t)\right]_{t} \right|\leq c(b,k)\dot{s}(t)$ and (\ref{zalA}) we have,  $\left[\vk(x,t)  \psi(t)\right]_{t} \in L^{1}(0,T;L^{2}(0,s(T)))$. Hence,
\[
\int_{0}^{T} \izsT I^{1-\alpha} \tilde{v}_{m}(x,t) \cdot \left[\vk(x,t)  \psi(t)\right]_{t} dx dt \to 
\int_{0}^{T} \izsT I^{1-\alpha} \tilde{v}(x,t) \cdot \left[\vk(x,t)  \psi(t)\right]_{t} dx dt.
\]
and
\eqq{
\int_{0}^{T}\izs \das v_{m}(x,t) \cdot \vk(x,t) dx \psi(t) dt
 \to 
-\int_{0}^{T} \izsT I^{1-\alpha}_{s} [v(x,t)-\tilde{v}_{0}(x)] \cdot \left[\vk(x,t)  \psi(t)\right]_{t} dx dt.}{wg}
Therefore, from (\ref{wh}) we obtain
\[
-\int_{0}^{T} \izsT I^{1-\alpha}_{s} [v(x,t)-\tilde{v}_{0}(x)] \cdot \left[\vk(x,t)  \psi(t)\right]_{t} dx dt + \int_{0}^{T} \izs v_{x}(x,t) \cdot \vkx(x,t) dx \psi(t) dt
\]
\eqq{
=\int_{0}^{T}\izs g(x,t) \vk(x,t) dx \psi(t) dt,
}{wi}
which finishes the proof of theorem~\ref{main}.

\end{proof}

\section{Appendix}\label{s:d}
In this section we collect facts, which we use in the previous parts of this paper.
We begin with the definition of fractional operators and some simple calculations.

\subsection{Basic facts about fractional operators}
For $\alpha\in (0,1)$ we define fractional integration of integrable function $f$
\[
\ia_{t_{0}}f(t)= \jga \int_{t_{0}}^{t} (t-\tau)^{\alpha-1} f(\tau) \dt.
\]
For simplicity we write $\ia f(t)=\ia_{0}f(t)$. One can check, see  \cite[formula (2.21)]{Samko},
that
\begin{equation}\label{r:zli}
 I^1 f = I^{1-\alpha}( I^\alpha f).
\end{equation}

For absolutely continuous function $f$ 
we define the Caputo fractional derivative,
\[
\da f(t) = \jgja \izt (t-\tau)^{-\alpha} \dot{f}(\tau) \dt.
\]
Recall that if $p<t$, then we have
\eqq{\int_{p}^{t} \tamj \tap   \dt=B(1-\alpha, \alpha),}{dodb}
where $B(x,y)$ denotes the Beta function. Furthermore,
\eqq{\ga \gja =B(1-\alpha,\alpha).}{dodc}
If $f$ is absolutely continuous, then by Fubini theorem and (\ref{dodb})  we have,
\eqq{\ia \da f(t)=f(t)-f(0).}{doda}

\medskip\noindent{\bf Lemma \ref{lemone}.} \ {\it
If $f$ is absolutely continuous, then $(\ia f)(t)\in AC[0,T]$ and $(\ia f)'(t)=\ia f'(t) + t^{\alpha-1}\frac{f(0)}{\ga}$.
}
\begin{proof}
Using the well-known characterization of absolutely continuous functions, we can write,
\[
(\ia f) (t) = (I^{\alpha+1}f')(t)+(\ia f(0))(t)=I(I^{\alpha}f')(t)+c_{0}f(0)t^{\alpha}.
\]
After taking the derivative of both sides we arrive at the desired statement.
\end{proof}

\subsection{Estimates}
We will establish two lemmas, which appears to be essential in the proof of Lemma~\ref{defX}.
\begin{lemma}
We assume that $v \in C([0,T];\rr^{m})$ and let $g_{1}(\tau) =\izta \tap \df (p,\tau) v(p)dp $.
Then, there exists $\tilde{g}_{1}$, an absolutely continuous representative of integrable function $g_1$.
\label{lemtwo}
\end{lemma}

\begin{proof}
It is sufficient to 
show that $g_{1} \in W^{1,1}(0,T)$, i.e.
\eqq{\int_{0}^{T} \varphi'(\tau) g_{1}(\tau) \dt = - \int_{0}^{T} \varphi(\tau) h_{1}(\tau) \dt \hd \hd \mbox{ for every }  \varphi \in C^{\infty}_{0}(0,T),}{oc}
where
\eqq{\izta [\tap \df(p,\tau)]_{\tau} v(p) dp=
h_{1}(\tau)
\in L^{1}(0,T).}{lb}
For this purpose, 
it is enough to show that
\eqq{\tau \mapsto \izta \tapj \df(p,\tau)  v(p) dp\in L^{1}(0,T),}{lc}
\eqq{ \tau \mapsto \izta \tap (\df(p,\tau))_{\tau}  v(p)dp\in L^{1}(0,T).}{ld}
From (\ref{ja}) we have
\[
\izT \left| \izta \tapj \df(p,\tau) v(p) dp \right| \dt \leq c_{0}\izT  \izta \tapj \dos(p) [s(\tau)-s(p)] dp  \dt.
\]
Using (\ref{ib}) and (\ref{od}), we may write that
\eqq{|s(\tau)-s(p)| \leq c_{0} p^{\gamma(\alpha-1)} |\tau - p|^{\gamma+(1-\gamma)\alpha} \hd \mbox{ for } p<\tau, \hd }{oe}
where $\gamma \in (0,1)$ will be chosen below.
Hence, using (\ref{zalA}), we can estimate the above integral by
\[
c_{0}\izT  \izta   (\tau - p )^{\gamma(1-\alpha)-1} p^{(\alpha-1)(\gamma+1)} dp  \dt,
\]
which is finite, provided  that $0<\gamma <\frac{\alpha}{1-\alpha}$. Thus,  (\ref{lc}) follows. Now, we shall show  (\ref{ld}). Using (\ref{jb}), we get
\begin{eqnarray*}
\izT \left| \izta \tap (\df(p,\tau))_{\tau}  v(p) dp \right| \dt &\leq& c_{0}  \izT \dot{s}(\tau)  \izta \tap \dos(p) dp  \dt
\\
 &\leq &c_{0}  \izT \tau^{\alpha-1}  \izta \tap p^{\alpha-1} dp  \dt
<\infty.
\end{eqnarray*}
In this way we proved (\ref{ld}) and  (\ref{lb}).

We are going to show (\ref{oc}). By Fubini Theorem we have,
\begin{equation}\label{410m}
 \int_{0}^{T} \varphi'(\tau) \izta \tap \df (p,\tau) v(p) dp  \dt =
\int_{0}^{T} v(p)\int_{p}^{T} \varphi'(\tau) \tap \df (p,\tau)   \dt dp.
\end{equation}
In order to integrate by parts, we have to show that  for each $p\in (0,T)$ function $\Psi_{p}$, defined below,
is absolutely continuous on $[0,T]$.
\eqq{
\Psi_{p}(\tau) = \left\{
\begin{array}{cll}
\tap \df(p, \tau) & \mbox{ for } & \tau >p,  \\
0 & \mbox{ for } & \tau \leq p.
\end{array} \right.
}{le}
For this purpose, we will show that  $\Psi_{p}'(\tau) \in L^{1}(0,T)$ for each $p \in (0,T)$.
From (\ref{ja}) and (\ref{od}) we infer,
\[
\int_{p}^{T}|(\tau-p)^{-\alpha-1}\widetilde{D}(p,\tau)|d\tau \leq c_{0}\dot{s}(p)p^{\alpha-1}\int_{p}^{T}(\tau-p)^{-\alpha}d\tau.
\]
Using (\ref{jb}) and (\ref{zalA}) we obtain
\[
\int_{p}^{T}|(\tau-p)^{-\alpha}\widetilde{D}(p,\tau)_{\tau}|d\tau \leq c_{0}\dot{s}(p)\int_{p}^{T}(\tau-p)^{-\alpha}\tau^{\alpha-1}d\tau.
\]
Thus, we conclude that $\Psi_{p}'(\tau)$ is integrable function for every fixed $p \in (0,T)$.
Hence, we may integrate by parts in (\ref{410m}) and (\ref{oc}) follows. As a result, $g_{1} \in W^{1,1}(0,T)$ and by \cite[Theorem~1, Ch. 4.9.1]{Evans}, 
$g_{1}$ is equal a.e. in $[0,T]$ to an absolutely continuous $\tilde{g}_{1}$.

Finally,  $\tilde{g}_{1}(0)=0$, because of  (\ref{ja}) and (\ref{ib}), for any sequence $\tau_{n}\searrow 0$, we get
\eqq{
|\tilde{g}_{1}(\tau_{n})|\leq c_{0} \int_{0}^{\tau_{n}}(\tau_{n}-p)^{-\alpha} \dos (p)[s(\tau_{n})-s(p)]dp \leq c_{0} [s(\tau_{n})-s(0)] \longrightarrow 0.
}{lf}
\end{proof}

\medskip\noindent{\bf Lemma \ref{lemtwod}.}\ {\it
We assume that $p^{1-\alpha}w(p)\in L^{\infty}(0,T)$ and let $g_{2}(\tau) =\izta \tap B (p,\tau) w(p) dp $. Then, $g_{2}=\tilde{g}_{2}$ a.e. on $[0,T]$, where  $\tilde{g}_{2} \in AC[0,T]$ and $\tilde{g}_{2}(0)=0$.}

\begin{proof}
As in the previous Lemma, we will show that $g_{2} \in W^{1,1}(0,T)$, i.e.
\eqq{\int_{0}^{T} \varphi'(\tau) g_{2}(\tau) \dt = - \int_{0}^{T} \varphi(\tau) h_{2}(\tau) \dt \hd \hd \mbox{ for }  \varphi \in C^{\infty}_{0}(0,T),}{noc}
where
\eqq{h_{2}(\tau)= \izta [\tap B(p,\tau)]_{\tau} w(p) dp\in L^{1}(0,T).}{nlb}
We will show that
\eqq{\tau \mapsto \izta \tapj B(p,\tau)  w(p)dp\in L^{1}(0,T),}{nlc}
\eqq{\tau \mapsto \izta \tap (B(p,\tau))_{\tau} w(p) dp\in L^{1}(0,T).}{nld}
From (\ref{eg}) and the assumption concerning $w$, we have,
\[
\izT \left| \izta \tapj B(p,\tau)  w(p)dp \right| \dt \leq c_{0}\izT  \izta \tapj p^{\alpha-1} [s(\tau)-s(p)] dp  \dt.
\]
We apply (\ref{oe}) and proceed as in the proof of Lemma~\ref{lemtwo} to deduce that the integral on the right-hand-side above is finite.

We are going to show  (\ref{nld}). Using (\ref{egt}) and the estimate for $w$ we get,
\begin{eqnarray*}
 \izT \left| \izta \tap (B(p,\tau))_{\tau} w(p)  dp \right| \dt
&\leq &c_{0}  \izT \dos(\tau)  \izta \tap p^{\alpha-1} dp  \dt\\
&=& c_{0} [s(T)-s(0)]<\infty.
\end{eqnarray*}
Thus, we are able to deduce (\ref{nld}) and  (\ref{nlb}).

It remains to show (\ref{noc}). By the Fubini theorem, we have
\[
\int_{0}^{T} \varphi'(\tau) \izta \tap B (p,\tau) dp w(p) \dt =
\int_{0}^{T} w(p)\int_{p}^{T} \varphi'(\tau) \tap B (p,\tau)   \dt dp.
\]
As in the proof of Lemma~\ref{lemtwo}, we will show that  for each $p\in (0,T)$ function, defined
below,
\eqq{
\Phi_{p}(\tau) = \left\{ \begin{array}{cll} \tap B(p, \tau) & \mbox{ for } & \tau >p  \\ 0 & \mbox{ for } & \tau \leq p \\ \end{array} \right.
}{nle}
has integrable derivative on $(0,T)$.
Making use of (\ref{eg}) and (\ref{od}) we get
\[
\int_{p}^{T}|(\tau-p)^{-\alpha-1}B(p,\tau)|d\tau \leq c_{0}p^{\alpha-1}\int_{p}^{T}(\tau-p)^{-\alpha}d\tau.
\]
Applying (\ref{egp}) and (\ref{zalA}) we obtain the following estimate
\[
\int_{p}^{T}|(\tau-p)^{-\alpha}B(p,\tau)_{\tau}|d\tau \leq c_{0} \int_{p}^{T}(\tau-p)^{-\alpha}\tau^{\alpha-1}d\tau.
\]
Thus, following the argument as in
the proof of Lemma~\ref{lemtwo}, we  obtain the claim of the Lemma~\ref{lemtwod}.
\end{proof}

Now, we will prove  estimates, which are crucial for proofs of Lemma~\ref{kont} and Lemma~\ref{konti}.

\begin{lemma}
We assume that $\alpha\in (0,1)$ and $s$ satisfies (\ref{zalA}). If function
$Q_{1,1}$ is defined in (\ref{qa}), $Q_{1,2}$ is given by (\ref{qd}),
then there exists a  constant $c_{0}=c_{0}(\alpha)$
such that the following inequalities hold,
\eqq{\lim_{t_{2} \rightarrow t_{1}^{+}} Q_{1,1}(t_{1},t_{2})=0 \hd \mbox{ for } \hd t_{1} \geq 0, }{qb}
\eqq{|Q_{1,1}(0, t)| \leq c_{0} [s(t)-s(0)],}{qc}
\eqq{|Q_{1,2}(t_{1}, t_{2})| \leq c_{0} |t_{2}-t_{1}|^{\alpha}.}{qe}
\label{estiQjj}
\end{lemma}

\begin{proof}
\textbf{Case of $Q_{1,1}$. } With the the Fubini theorem of can write
\begin{eqnarray*}
Q_{1,1}(t_{1},t_{2}) &=&t^{1-\alpha }_{2} \iztjd \tamjd \izta \tapj \int_{p}^{\tau} \dot{s}(q) dq p^{\alpha-1}dp \dt
\\
&=&t^{1-\alpha }_{2} \iztjd \tamjd \izta \dot{s}(q) \int_{0}^{q} \tapj    p^{\alpha-1} dp dq \dt.
\end{eqnarray*}
Applying substitution $\tau = t_{2}a$ and then $p=t_{2}ab$ we get
\begin{eqnarray*}
Q_{1,1}(t_{1},t_{2}) &=&
t_{2} \izjjd (1-a)^{\alpha-1}  \int_{0}^{t_{2}a} \dot{s}(q) \int_{0}^{q} (t_{2}a-p)^{-\alpha-1}    p^{\alpha-1} dp dq da
\\
&= &\izjjd (1-a)^{\alpha-1}  a^{-1}\int_{0}^{t_{2}a} \dot{s}(q) \int_{0}^{\frac{q}{t_{2}a}} (1-b)^{-\alpha-1}    b^{\alpha-1} db dq da
\\
&= &\izjjd (1-a)^{\alpha-1}  a^{-1}\int_{0}^{t_{2}a} \dot{s}(q) \left. \left[ \alpha^{-1}(1-b)^{-\alpha}    b^{\alpha} \right] \right|_{b=0}^{b=\frac{q}{t_{2}a}} dq da
\\
&= &\izjjd (1-a)^{\alpha-1}  a^{-1} \int_{0}^{t_{2}a}  q^{\alpha} \dot{s}(q) (t_{2}a-q)^{-\alpha} dq da.
\end{eqnarray*}
Making use of the Fubini theorem again, we have
\begin{eqnarray*}
 Q_{1,1}(t_{1},t_{2}) &=&
 \int_{0}^{t_{1}}   \dot{s}(q) q^{\alpha}\izjjd (1-a)^{\alpha-1}  a^{-1}  (t_{2}a-q)^{-\alpha} da dq
\\
&&+\iztjd  \dot{s}(q)  q^{\alpha} \int_{\frac{q}{t_{2}}}^{1} (1-a)^{\alpha-1}  a^{-1} (t_{2}a-q)^{-\alpha} da dq\\
&\equiv & Q_{1,1,1}(t_{1},t_{2})+Q_{1,1,2}(t_{1},t_{2}).
\end{eqnarray*}
Obviously, $Q_{1,1,1}(0,t_{2})=0$. Therefore, we have to prove the continuity of $Q_{1,1,1}(t_{1}, t_{2})$ for $t_{1}>0$. After having substituted $b=\frac{1-a}{a-q/t_{2}}$  and  $a=(q/t_{2})b$, we get
\[
Q_{1,1,1}(t_{1},t_{2}) = \iztj \dos (q)  (q/t_{2})^{\alpha}\int_{0}^{\frac{t_{2}-t_{1}}{t_{1}-q} } \frac{b^{\alpha-1}}{1+\frac{q}{t_{2}}b} db dq =
\iztj \dos (q)  \int_{0}^{\frac{t_{2}-t_{1}}{t_{2}}\frac{q}{t_{1}-q}} \frac{a^{\alpha-1}}{1+a} da dq .
\]
We shall show that above integral is arbitrary small if $t_{2}$ is sufficiently close to $t_{1}$. Indeed, if we fix $\ep >0$, then there exists $c\in (0,t_{1})$ such that
\[
\int_{c}^{t_{1}} \dos (q)  \int_{0}^{\frac{t_{2}-t_{1}}{t_{2}}\frac{q}{t_{1}-q}} \frac{a^{\alpha-1}}{1+a} da dq  < \frac{\ep}{2},
\]
because we are dealing with integrable functions.
Then,
\begin{eqnarray*}
\int_{0}^{c} \dos (q)  \int_{0}^{\frac{t_{2}-t_{1}}{t_{2}}\frac{q}{t_{1}-q}} \frac{a^{\alpha-1}}{1+a} da dq  &\leq& \int_{0}^{c} \dos (q)  \int_{0}^{\frac{t_{2}-t_{1}}{t_{2}}\frac{c}{t_{1}-c}} \frac{a^{\alpha-1}}{1+a} da dq\\
&\leq &[s(c)-s(0)]\int_{0}^{\frac{t_{2}-t_{1}}{t_{2}}\frac{c}{t_{1}-c}} \frac{a^{\alpha-1}}{1+a} da <\frac{\ep}{2},
\end{eqnarray*}
if $t_{2}$ is sufficiently close to $t_{1}$. Therefore,
\[
\lim_{t_{2}\rightarrow t_{1}^{+}}Q_{1,1,1}(t_{1},t_{2})=0.
\]
In order to estimate $Q_{1,1,2}$, we notice that
\[
Q_{1,1,2}(t_{1},t_{2})=
\iztjd \dot{s}(q) G_{1}\left(\frac{q}{t_{2}}\right) dq,
\]
where
\[
G_{1}(x)=x^{\alpha}\int_{x}^{1}(1-a)^{\alpha-1}  a^{-1}  ( a-x)^{-\alpha} da.
\]
The last integral can be easily evaluated. Substituting $b=\frac{1-a}{a-x}$ and next $a=xb$, we get
\eqq{
G_{1}(x) =  x^{\alpha}\int_{0}^{\infty} \frac{b^{\alpha-1}}{1+bx} db
=\int_{0}^{\infty} \frac{a^{\alpha-1}}{1+a} da= \frac{\pi}{\sin{\pi \alpha}}
 .}{oszag}
Hence,
\[
Q_{1,1,2}(t_{1},t_{2}) \leq  \sup_{x\in (0,1)}{G_{1}(x)} \iztjd \dot{s}(q) =c_{0} [s(t_{2})-s(t_{1})].
\]
Thus,
\[
\lim_{t_{2}\rightarrow t_{1}}Q_{1,1}(t_{1},t_{2})= \lim_{t_{2}\rightarrow t_{1}}(Q_{1,1,1}(t_{1},t_{2})+Q_{1,1,2}(t_{1},t_{2}))=0
\]
and
\[
|Q_{1,1}(0,t)|=|Q_{1,1,2}(0,t)| \leq c_{0} [s(t)-s(0)].
\]
Now, we shall deal with $Q_{1,2}$. We can write
 \[
Q_{1,2}(t_{1},t_{2}) = c_{0} t^{1-\alpha}_{2} \iztjd \tamjd \tau^{\alpha-1}  \dt.
\]
Applying substitution $\tau = t_{2}-(t_{2}-t_{1})a$, we get
\[
c_{0} (t_{2}-t_{1})^{\alpha } \izj a^{\alpha-1} \left(1-\frac{t_{2}-t_{1}}{t_{2}}a\right)^{\alpha-1}da \leq c_{0} (t_{2}-t_{1})^{\alpha } \izj a^{\alpha-1} \left(1-a\right)^{\alpha-1}da,
\]
and the proof is finished.
\end{proof}

\begin{lemma}
Let us assume that $\alpha\in (0,1)$ and $s$ satisfies (\ref{zalA}). If function $Q_{2,1}$ (resp. $Q_{2,2}$) is defined by (\ref{qf}) (resp. by (\ref{qg})), then
\eqq{\lim_{t_{2} \rightarrow t_{1}^{+}} Q_{2,i}(t_{1},t_{2})=0 \hd \mbox{ for } \hd t_{1} > 0, \hd i=1,2. }{qqf}
\label{estiQdd}
\end{lemma}

\begin{proof}
We first deal with $Q_{2,1}$,
\[
Q_{2,1}(t_{1},t_{2})= t^{1-\alpha}_{2} \iztj [\tamjj -\tamjd] \izta   \tapj p^{\alpha-1} \int_{p}^{\tau} \dos(q) dqdp  \dt.
\]
Applying the Fubini theorem and substitutions $\tau = t_{1} a$ and next $p=t_{1}ab$, we see that
\begin{eqnarray*}
 Q_{2,1}(t_{1},t_{2})&=&
 t^{1-\alpha}_{2} \iztj [\tamjj -\tamjd] \izta \dos(q) \int_{0}^{q}  \tapj p^{\alpha-1}   dpdq  \dt
\\
&= &t^{1-\alpha}_{2} t_{1} \izj [(t_{1}-t_{1}a)^{\alpha-1} -(t_{2}-t_{1}a)^{\alpha-1}] \int^{t_{1}a}_{0} \dos(q) \int_{0}^{q}  (t_{1}a-p)^{-\alpha-1} p^{\alpha-1}   dpdq  da
\\
&=&t^{1-\alpha}_{2} t_{1}^{\alpha-1} \izj [(1-a)^{\alpha-1} -(\frac{t_{2}}{t_{1}}-a)^{\alpha-1}]a^{-1} \int^{t_{1}a}_{0} \dos(q) \int_{0}^{\frac{q}{t_{1}a}}  (1-b)^{-\alpha-1} b^{\alpha-1}   dbdq  da
\\
&=&\alpha \left( \frac{t_{2}}{t_{1}}\right)^{1-\alpha} \izj [(1-a)^{\alpha-1} -(\frac{t_{2}}{t_{1}}-a)^{\alpha-1}] a^{-1}\int^{t_{1}a}_{0} \dos(q)   (1-b)^{-\alpha} b^{\alpha}\Big|_{b=0}^{b=\frac{q}{t_{1}a}}   dq  da
\\
&=&\alpha \left( \frac{t_{2}}{t_{1}}\right)^{1-\alpha} \izj [(1-a)^{\alpha-1} -(\frac{t_{2}}{t_{1}}-a)^{\alpha-1}] a^{-1}\int^{t_{1}a}_{0} \dos(q)   \left(\frac{t_{1}a}{q}-1\right)^{-\alpha} dq  da.
\end{eqnarray*}
Using again the Fubini theorem, we get,
\[
Q_{2,1}(t_{1},t_{2}) =\alpha \left( \frac{t_{2}}{t_{1}}\right)^{1-\alpha} \iztj  \dos(q)
 \left( \frac{q}{t_{1}}\right)^{\alpha} \int^{1}_{\frac{q}{t_{1}}} [(1-a)^{\alpha-1} -(\frac{t_{2}}{t_{1}}-a)^{\alpha-1}] a^{-1}    \left(a-\frac{q}{t_{1}}\right)^{-\alpha} da  dq.
\]
By (\ref{oszag}) we arrive at
\begin{eqnarray*}
Q_{2,1}(t_{1},t_{2})&=&
\alpha\iztj \int^{1}_{\frac{q}{t_{1}}}  \dos(q)
 \left( \frac{q}{t_{1}}\right)^{\alpha}  (1-a)^{\alpha-1}  a^{-1}    \left(a-\frac{q}{t_{1}}\right)^{-\alpha} da  dq\\
& = &\alpha\iztj \dos(q) G_{1}\left(\frac{q}{t_{1}} \right) dq \leq c_{0} [s(t_{1})-s(0)]<\infty.
\end{eqnarray*}
Thus, we can apply Lebesgue dominated convergence theorem and we see that
\[
\lim_{t_{2}\rightarrow t_{1}^{+}}Q_{2,1}(t_{1},t_{2})=0.
\]
Now, we estimate $Q_{2,2}$. We see,
\[
|Q_{2,2}(t_{1},t_{2})| = c_{0} t^{1-\alpha}_{2} \iztj [\tamjj -\tamjd] \tau^{\alpha-1} \dt.
\]
Thus, applying Lebesgue dominated convergence theorem, we get
$\lim_{t_{2}\rightarrow t_{1}^{+}}Q_{2,2}(t_{1},t_{2})=0$ and the proof is finished.
\end{proof}

\begin{lemma}
Let us assume that $\alpha \in (0,1)$, $f\in AC[0,T]$ and $t^{1-\alpha}f'\in L^{\infty}(0,T)$. Then,
\eqq{|t^{1-\alpha}_{2} (\ia f')(t_{2})-t^{1-\alpha}_{1} (\ia f')(t_{1})|\leq c_{0} \| t^{1-\alpha} f' \|_{L^{\infty}(0,T)} |t_{2}-t_{1}|^{\alpha},}{dode}
where $c_{0} $ depends only on $\alpha$. In particular, $t\mapsto t^{1-\alpha} (\ia f')(t) \in C^{0,\alpha}([0,T])$ and  $\da f\in C^{0,1-\alpha}_{loc}((0,T])$.
\label{lemthree}
\end{lemma}

\begin{proof}
For the sake of the simplicity of notation in this proof, we write $\| f \|:=\| t^{1-\alpha }f' \|_{L^{\infty}(0,T)}$ and assume that  $t_{2}>t_{1}$. Then we have
\begin{eqnarray*}
&&\ga | t_{2}^{1-\alpha} \ia f'(t_{2})-t_{1}^{1-\alpha} \ia f'(t_{1}) |
\\
&\leq &|t_{2}^{1-\alpha}- t_{1}^{1-\alpha}| \left| \int_{0}^{t_{2}} (t_{2}-\tau)^{\alpha-1}f'(\tau) \dt \right|
\\
& &+t_{1}^{1-\alpha} \left| \int_{0}^{t_{2}} (t_{2}-\tau)^{\alpha-1}f'(\tau) \dt -\int_{0}^{t_{1}} (t_{1}-\tau)^{\alpha-1}f'(\tau) \dt \right|\\
&\equiv&  B_{1}(t_{1},t_{2})+B_{2}(t_{1},t_{2}),
\end{eqnarray*}
and $B_{2}(0,t_{2})=0$.  Using the obvious inequality
$|f'(s)|\leq \nf s^{\alpha-1}$ 
and the substitution $\tau=t_{2}s$, we can write
\begin{eqnarray*}
B_{1}(t_{1},t_{2}) &\leq &\nf |t_{2}^{1-\alpha}- t_{1}^{1-\alpha}|  \int_{0}^{t_{2}} (t_{2}-\tau)^{\alpha-1}\tau^{\alpha-1} \dt\\
&=&\nf |t_{2}^{1-\alpha}- t_{1}^{1-\alpha}| t_{2}^{2\alpha-1}  \int_{0}^{1} (1-s)^{\alpha-1}s^{\alpha-1} ds
=c_{0} \nf (\tadj) h\left( \frac{t_{1}}{t_{2}}\right) ,
\end{eqnarray*}
 where $h(x)=\frac{1-x^{1-\alpha}}{1-x^{\alpha}}$. Since
 $|t_{2}^{\alpha}-t_{1}^{\alpha}| \leq |t_{2}-t_{1}|^{\alpha}$ and function $h(x)$ is bounded on  $[0,1]$,
we get
\[
B_{1} (t_{1},t_{2})\leq c_{0 } \nf |t_{2}-t_{1}|^{\alpha}.
\]
Now, we estimate $B_{2}$. 
In this case $t_{1}>0$.   We use substitution $s=t_{1}-\tau$ and decompose $B_{2}$ as follows,
\begin{eqnarray*}
&&B_{2}(t_{1},t_{2})\\
&=&  t_{1}^{1-\alpha} \left| \int_{0}^{t_{2}} (t_{2}-\tau)^{\alpha-1}f'(\tau) \dt -\int_{0}^{t_{1}} (t_{1}-\tau)^{\alpha-1}f'(\tau) \dt \right|
\\
&=&t_{1}^{1-\alpha} \left| -\int_{t_{1}}^{t_{1}-t_{2}} ((t_{2}-t_{1})+s)^{\alpha-1}f'(t_{1}-s) ds +\int^{0}_{t_{1}} s^{\alpha-1}f'(t_{1}-s) ds \right|
\\
&=&t_{1}^{1-\alpha} \left| \int^{0}_{t_{1}-t_{2}} ((t_{2}-t_{1})+s)^{\alpha-1}f'(t_{1}-s) ds +\int_{0}^{t_{1}} \left[ ((t_{2}-t_{1})+s)^{\alpha-1}-s^{\alpha-1} \right] f'(t_{1}-s) ds \right|.
\end{eqnarray*}
Applying  $|f'(s)|\leq \nf s^{\alpha-1}$ 
we reach the following estimate
\begin{eqnarray*}
B_{2}(t_{1},t_{2}) &\leq& \nf t_{1}^{1-\alpha}  \int^{0}_{t_{1}-t_{2}} ((t_{2}-t_{1})+s)^{\alpha-1}(t_{1}-s)^{\alpha-1} ds
\\
&&+\nf t_{1}^{1-\alpha}\int_{0}^{t_{1}} \left[ s^{\alpha-1}-((t_{2}-t_{1})+s)^{\alpha-1} \right] (t_{1}-s)^{\alpha-1} ds
\\
&\equiv &\nf  B_{2,1}(t_{1},t_{2})+\nf B_{2,2}(t_{1},t_{2}).
\end{eqnarray*}
First, we deal with  $B_{2,2}$. After substitution $(\tdj)t=s$, we get
\begin{eqnarray*}
B_{2,2}(t_{1},t_{2})&=& t_{1}^{1-\alpha} \tdjaj \int_{0}^{t_{1}} \left[ \left(\frac{ s}{\tdj}\right)^{\alpha-1} -(1+\frac{s}{\tdj})^{\alpha-1}\right] (t_{1}-s)^{\alpha-1} ds
\\
&=&t_{1}^{1-\alpha} \tdjaj \int_{0}^{\frac{t_{1}}{\tdj}} \left[ t^{\alpha-1} -(1+t)^{\alpha-1}\right] (t_{1}-(\tdj)t)^{\alpha-1} (\tdj)dt
\\
&= &\tdja \int_{0}^{\frac{t_{1}}{\tdj}} \left[ t^{\alpha-1} -(1+t)^{\alpha-1}\right] (1-\frac{\tdj}{t_{1}}t)^{\alpha-1} dt.
\end{eqnarray*}
We write $a=\frac{t_{1}}{\tdj}>0$. 
We have to  estimate the following integral,
\eqq{g(a)=\int_{0}^{a} \left[ t^{\alpha-1} -(1+t)^{\alpha-1}\right] (1-\frac{t}{a})^{\alpha-1} dt, }{d}
independently of $a\in (0,\infty)$.  We have to consider a number of cases. First, we assume that  $a\in (0,2]$. Then, using substitution $as=a-t$, we get,
\begin{eqnarray*}
g(a)&=& a^{1-\alpha} \int_{0}^{a} \left[ t^{\alpha-1} -(1+t)^{\alpha-1}\right] (a-t)^{\alpha-1} dt \leq a^{1-\alpha} \int_{0}^{a}  t^{\alpha-1}  (a-t)^{\alpha-1} dt
\\
&=& a^{\alpha} B(1-\alpha, 1-\alpha) \leq 2^{\alpha} B(1-\alpha, 1-\alpha) .
\end{eqnarray*}
If  $a\in (2, \infty)$, we can write
\[
g(a)= \left( \int_{0}^{1}+\int_{1}^{\frac{a}{2}}+\int_{\frac{a}{2}}^{a} \right) \left[ t^{\alpha-1} -(1+t)^{\alpha-1}\right] (1-\frac{t}{a})^{\alpha-1} dt.
\]
Then, applying 
$(1-\frac{t}{a})^{\alpha-1} \leq (1-t)^{\alpha-1}$, we get
\begin{eqnarray*}
\int_{0}^{1} \left[ t^{\alpha-1} -(1+t)^{\alpha-1}\right] (1-\frac{t}{a})^{\alpha-1} dt
&\leq& \int_{0}^{1} \left[ t^{\alpha-1} -(1+t)^{\alpha-1}\right] (1-t)^{\alpha-1} dt \\
&=& g(1) \leq B(1-\alpha, 1-\alpha).
\end{eqnarray*}
Next, by the 
mean theorem, we obtain
\eqq{t^{\alpha-1} -(1+t)^{\alpha-1} \leq (1-\alpha)t^{\alpha-2} .}{e}
Hence, 
since $a-t\ge \frac{a}{2}$, we obtain
\begin{eqnarray*}
&&\int_{1}^{\frac{a}{2}} \left[ t^{\alpha-1} -(1+t)^{\alpha-1}\right] (1-\frac{t}{a})^{\alpha-1} dt \leq (1-\alpha)\int_{1}^{\frac{a}{2}} t^{\alpha-2} (1-\frac{t}{a})^{\alpha-1} dt
\\
&=&(1-\alpha) a^{1-\alpha}\int_{1}^{\frac{a}{2}} t^{\alpha-2} (
{a-t}
)^{\alpha-1} dt \leq (1-\alpha)2^{1-\alpha} \int_{1}^{\frac{a}{2}} t^{\alpha-2} dt=
2^{\alpha-1}[1-(\frac{a}{2})^{\alpha-1}]\\& \leq& 2^{\alpha-1}.
\end{eqnarray*}
Using again (\ref{e}), we deduce,
\begin{eqnarray*}
&&\int_{\frac{a}{2}}^{a}  \left[ t^{\alpha-1} -(1+t)^{\alpha-1}\right] (1-\frac{t}{a})^{\alpha-1} dt \leq (1-\alpha)\int_{\frac{a}{2}}^{a}   t^{\alpha-2}  (1-\frac{t}{a})^{\alpha-1} dt
\\
&\leq &2^{2-\alpha} a^{\alpha-2}\int_{\frac{a}{2}}^{a}    (1-\frac{t}{a})^{\alpha-1} dt
= \frac{4^{1-\alpha}}{\alpha} a^{\alpha-1}\leq \frac{2^{1-\alpha}}{\alpha}.
\end{eqnarray*}
Therefore, $\sup\limits_{a\in (0,\infty)} g(a) \leq c_{0}$ and as a result,
\[
B_{2,2}(t_{1},t_{2}) \leq c_{0} |\tdj|^{\alpha}.
\]
Now, we turn to 
$B_{2,1}$. Here, 
$s=t_1 -\tau$ is negative, hence from 
$(1-\frac{s}{t_{1}})^{\alpha-1}\leq 1$ we get
\begin{eqnarray*}
B_{2,1}(t_{1},t_{2})&=&t_{1}^{1-\alpha}  \int^{0}_{t_{1}-t_{2}} ((t_{2}-t_{1})+s)^{\alpha-1}(t_{1}-s)^{\alpha-1} ds
\\&= &\tdjaj \int^{0}_{t_{1}-t_{2}} (1+\frac{s}{\tdj})^{\alpha-1}(1-\frac{s}{t_{1}})^{\alpha-1} ds
\\
 &\leq&
\tdjaj \int^{0}_{t_{1}-t_{2}} (1+\frac{s}{\tdj})^{\alpha-1} ds=\frac{1}{\alpha} (\tdj)^{\alpha}.
\end{eqnarray*}
Thus,
\[
B_{2} (t_{1},t_{2})\leq c_{} \nf (\tdj)^{\alpha},
\]
and the proof of (\ref{dode}) is completed.
\no If we set $t^{1-\alpha}(\ia f')(t)_{|t=0}=0$, then (\ref{dode}) means that $t^{1-\alpha}(\ia f')(t) \in C^{0, \alpha}([0,T])$ and then $(\ia f')(t) \in C^{0, \alpha}_{loc}((0,T])$. To  finish the proof it is enough to notice that $D^{1-\alpha}f=\ia f'$.
\end{proof}

\end{document}